\documentclass{siamonline190516}

\usepackage{float}
\usepackage{amsfonts}
\usepackage{graphicx}
\usepackage{overpic}
\usepackage{multirow}
\usepackage{subcaption}
\usepackage{mathtools}
\usepackage{psfrag}
\usepackage{dirtytalk}
\usepackage{epstopdf}
\usepackage{algorithm}
\usepackage{algpseudocode}

\usepackage[shortlabels]{enumitem}
\usepackage{xspace}
\usepackage{hyperref}
\usepackage{cleveref}
\crefname{algorithm}{algorithm}{algorithms}
\Crefname{algorithm}{Algorithm}{Algorithms}
\ifpdf
  \DeclareGraphicsExtensions{.eps,.pdf,.png,.jpg,.jpeg}
\else
  \DeclareGraphicsExtensions{.eps}
  \fi

\graphicspath{{./figures/}}

\usepackage{enumitem}
\setlist[enumerate]{leftmargin=.5in}
\setlist[itemize]{leftmargin=.5in}

\newsiamthm{claim}{Claim}
\newsiamremark{conjecture}{Conjecture}
\newsiamremark{remark}{Remark}
\newsiamremark{example}{Example}
\newsiamremark{hypothesis}{Hypothesis}
\newsiamremark{problem}{Problem}
\newsiamremark{assumption}{Assumption}

\usepackage{etoolbox}
\AtEndEnvironment{remark}{\null\hfill$\Diamond$}

\usepackage{pict2e}

\DeclareRobustCommand{\lowerrighttriangle}{%
  \begingroup
  \setlength{\unitlength}{1ex}%
  \begin{picture}(1,1)
  \polyline(1,0)(0,0)(0,1)(1,0)(.5,0)
  \end{picture}%
  \endgroup
}

\newcommand{\TT}{\mathcal{T}_{\:\lowerrighttriangle}}

\newcommand{\mcl}{\mathcal}
\newcommand{\wt}[1]{\widetilde{#1}}

\newcommand{\mbb}{\mathbb}

\newcommand{\dd}{{\rm d}}

\newcommand{\eps}{\epsilon}

\newcommand{\vrho}{\varrho}

\newcommand{\X}{\mcl{X}}
\newcommand{\Y}{\mcl{Y}}
\newcommand{\Z}{\mcl{Z}}

\newcommand{\B}{\mcl{B}}

\newcommand{\U}{\mcl{U}}

\newcommand{\V}{{\mcl{V}}}
\newcommand{\W}{{\mcl{W}}}

\newcommand{\Lip}{{\rm Lip}}
\newcommand{\Id}{{\rm Id}}

\newcommand{\PP}{\mbb{P}}

\newcommand{\R}{\mbb{R}}

\newcommand{\Law}{{\rm Law}\xspace}

\definecolor{darkred}{rgb}{.7,0,0}

\definecolor{darkgreen}{rgb}{.15,.55,0}

\definecolor{darkblue}{rgb}{0,0,0.7}

\theoremstyle{plain}
\theoremheaderfont{\normalfont\sffamily}
\theorembodyfont{\normalfont\itshape}
\theoremseparator{.}
\theoremsymbol{}

\usepackage{amsopn}

\DeclareMathOperator*{\argmin}{arg\,min}

\DeclareMathOperator*{\st}{s.t.}

\reversemarginpar
\usepackage[colorinlistoftodos,prependcaption,textsize=tiny]{todonotes}

\headers{Conditional Optimal Transport on Function Spaces}{B. Hosseini, A.W. Hsu, A. Taghvaei}

\title{Conditional Optimal Transport on Function Spaces}

\author{ Bamdad Hosseini\thanks{University of Washington, Seattle, WA 98195, USA
  (\email{bamdadh@uw.edu}, \email{owlx@uw.edu}, \email{amirtag@uw.edu})} 
  \and Alexander W. Hsu\footnotemark[1] \and Amirhossein Taghvaei \footnotemark[1] 
  }

\ifpdf
\hypersetup{
  pdftitle={Conditional OT on Function Spaces},
  pdfauthor={B. Hosseini, A. Hsu, A. Taghvaei}
}
\fi

\begin{document}
\newcommand{\RefM}{\eta}
\newcommand{\TargetM}{\nu}

\maketitle

\begin{abstract}
 We present a systematic study of conditional triangular transport maps 
 in function spaces
 from the perspective of  
 optimal transportation and with a view towards amortized Bayesian inference. 
 More specifically, we develop a theory of constrained optimal transport problems
 that describe block-triangular Monge maps that characterize 
conditional measures along with their Kantorovich relaxations. This generalizes 
the theory of optimal triangular transport to separable infinite-dimensional 
function spaces with general cost functions. We further tailor our results 
to the case of Bayesian inference problems and obtain regularity estimates on 
the conditioning maps from the prior to the posterior. Finally, we present numerical 
experiments that demonstrate the computational applicability of our theoretical 
results for amortized and likelihood-free inference of functional parameters.  
\end{abstract}



\section{Introduction and setup}\label{sec:introduction}

\sloppy

This article presents a systematic study of conditional optimal transport (OT) on separable 
infinite-dimensional function spaces towards likelihood-free and amortized 
Bayesian inference. Our analysis gives an OT characterization 
of (block) triangular maps leading to algorithms for the computation of 
such maps for  PDE inverse problems. 

Sampling/simulation of infinite or high-dimensional probability measures is a fundamental task in  data science, computational statistics, and uncertainty quantification (UQ), with
Markov chain Monte Carlo (MCMC) 
\cite{robert1999monte, stuart-acta-numerica, stuart-mcmc, hairer2014spectral, cui2016dimension, beskos-geometric-mcmc, betancourt2017geometric, durmus2017nonasymptotic, hosseini2019two, garbuno2020interacting, hosseini2023spectral} and variational inference (VI) \cite{blei2017variational, zhang2018advances, fox2012tutorial} being the 
most broadly used approaches for this task.
In recent years a lot of interest has been generated around 
the family of transport-based sampling algorithms \cite{marzouk2016sampling}
due to the rise of generative models \cite{ng2002discriminative, jebara2012machine} such as generative adversarial networks (GANs) \cite{goodfellow2016nips,goodfellow2016deep, goodfellow2020generative},
normalizing flows (NFs) \cite{rezende2015variational, kobyzev2020normalizing, papamakarios2019normalizing}, and most recently, diffusion models (DMs)
\cite{sohl2015deep, song2020improved, kingma2021variational, cao2022survey, albergo2023stochastic}.  
Broadly speaking, these transport-based algorithms proceed as follows: 
given Hilbert spaces $\U, \V$, a {\it target} Borel probability 
measure $\TargetM \in \PP(\U)$, and a {\it reference} Borel probability measure 
$\RefM \in \PP(\V)$, a (parametric and possibly stochastic) transport map $T^\star$ is computed
such that $T^\star\# \RefM \approx \nu$. The reference $\RefM$ is 
often chosen to be trivial to simulate, such as a Gaussian measure, so that 
 $\TargetM$-samples can be (approximately) generated
 by drawing $v\sim \RefM$ and then evaluating $T^\star(v)$. The map $T^\star$ is often 
 computed by minimizing  a statistical divergence 
 between the pushforward $T^\star\#\RefM$ and the target $\TargetM$ 
 or their empirical approximations. 
 To this end, such transport-based algorithms can be categorized under VI  as minimum divergence estimators \cite{pardo2018statistical}.

 Naturally, this line of thinking leads one to the idea of characterizing 
 $T^\star$ as an OT/Monge map pushing $\eta$ to $\nu$ \cite{villani-OT}, connecting the fields of generative modeling and sampling 
 to computational OT \cite{peyre2019computational}. Exploring this connection has become an 
 area of intense interdisciplinary research 
 with ideas from OT being used in generative modelling 
 \cite{genevay2016stochastic, arjovsky2017wasserstein, gulrajani2017improved, makkuva2020optimal, onken2021ot, rout2021generative, uscidda2023monge} and, conversely, Machine Learning (ML) technology inspiring  algorithms for OT
 \cite{daniels2021score, korotin2021neural, amos2022amortizing, coeurdoux2022learning, korotin2022kernel, korotin2023neural}. 
 We should note that the overwhelming majority of contributions in the aforementioned areas are focused on the finite-dimensional setting  
 with the exception of  \cite{rahman2022generative, kerrigan2023diffusion, baptista-function-space-score, MGAN} that 
  focus specifically on infinite-dimensional  spaces.

Our goal in this paper is to study a particular type of transport that 
is  aimed at the characterization of conditional measures. We pose the following problem:

\begin{problem}[Conditional transport]\label{prob:conditional-transport}
Consider separable Hilbert spaces $\Y, \U, \V$ 
and Borel probability measures 
$\TargetM \in \PP(\Y \times \U)$ and 
$\RefM_\V \in \PP(\V)$. Let $\nu(\cdot \mid y)$ denote conditional of $\nu$ on the $y$
variable. 
Find a transport map $T^\star_\U: \Y \times \V \to \U$ such that 
$T^\star_\U(y, \cdot)\# \eta_\V = \nu( \cdot \mid y)$. 
\end{problem}
Here the measure $\eta_\V$ is a reference measure on $\V$ that is assumed to be 
easy to simulate. Assuming that we could solve the above problem then we obtain 
a simple recipe for the simulation of samples not from $\TargetM$ but any of its conditionals 
$\TargetM( \cdot \mid y)$ for a.e. choice of $y$ since we simply have to draw $v \sim \eta_\V$ and evaluate  $T^\star_\U(y, v) \sim \nu( \cdot \mid y)$. Such a strategy falls under the category of 
{\it amortized inference} \cite{siahkoohi2023reliable, zhang2018advances}, where 
an upfront cost is paid for the calculation of $T^\star_\U$ which can later generate 
conditional samples from $\nu(\cdot \mid y)$ for any choice of $y$. Such an strategy is 
inherently different from more traditional approaches such as MCMC which directly target 
$\nu(\cdot \mid y)$ for a single choice of $y$.

Our motivation for considering \Cref{prob:conditional-transport} is the solution of 
Bayesian inverse problems (BIPs) involving PDEs. 
 For example, consider the Darcy flow PDE that arises
 in various applications such as the
  subsurface flows \cite{iglesias-subsurface}  
 or electrical impedance tomography \cite{somersalo}:
 \begin{equation}\label{Darcy-PDE}
     - {\rm div} \: \exp( u^\dagger )  \: \nabla p 
     = f \quad {\rm in} \quad \Omega \qquad {\rm and} \qquad
      p = 0 \quad {\rm on} \quad \partial \Omega.
 \end{equation}
 Here $\Omega \subset \R^d$ is a smooth domain with boundary $\partial \Omega$, 
  $u^\dagger$ is a smooth field that models the permeability of the rock,
  $p$ is the pressure field, and $f$ is the source term. Suppose we have
  limited and noisy measurements $y = ( p(x_1), \dots, p(x_m)) + \eps \in \R^m$ of the 
  pressure field $p$ at a set of 
  points $\{ x_1, \dots, x_m \} \subset \Omega $
  and with Gaussian noise $\eps \sim N( 0, \sigma^2 I)$. Then our goal is to 
  estimate the permeability field $u^\dagger$ given the data $y$. 
  Employing the Bayesian perspective, we choose a prior measure $\mu \in \PP(\U)$, 
  e.g. a Gaussian random field, 
  and model $u^\dagger$ with the posterior measure $\nu( \cdot \mid y) \in \PP(\U)$ 
  defined via Bayes' rule \cite{stuart-acta-numerica}: 
  \begin{equation}\label{Bayes-rule-generic}
      \frac{\dd \nu( \cdot \mid y)}{\dd \mu}(u) 
      = \frac{1}{Z(y)} \exp \left( - \Phi(u ; y) \right),
      \quad Z(y) := \int_\U \exp( - \Phi(u; y) \dd \mu(u),
  \end{equation}
  with $\Phi$ denoting the negative log-likelihood of $y$ given $u$; 
   see \eqref{darcy-likelihood} for its expression. 
  The standard approach for sampling $\nu(\cdot \mid y)$
  is to use function space MCMC algorithms such as preconditioned-Crank-Nicolson (pCN) 
  or Metropolis Adjusted Langevine (MALA) \cite{stuart-mcmc} but 
  these algorithms require long sequential runs for each value of $y$. 
  However, by solving \Cref{prob:conditional-transport} we can obtain an
  approximation $\widehat{T}_\U$ to the map $T^\star_\U$ such that $\widehat{T}_\U (y, \cdot) \# \mu \approx \nu(\cdot \mid y)$. We use this 
  approach in \Cref{subsec:darcy} and compare 
  the performance of a triangular map appraoch 
  to that of MCMC; see in particular \Cref{fig:darcy-summary-stats}. The advantage of this 
  approach is that the map is only approximated once 
  and can simply be evaluated for any value of $y$ to generate posterior samples, 
  for comparison, an MCMC chain needs to be run to convergence for every new value of $y$.

In the rest of this section we introduce our setup and a summary of our contributions 
in a simplified setting in \Cref{subsec:main-contributions} followed by a 
review of relevant literature in \Cref{subsec:rev:lit} and an outline of the 
article in \Cref{subsec:outline}.
  

\subsection{Setup and summary of contributions}\label{subsec:main-contributions}

Our primary focus is the systematic study of \Cref{prob:conditional-transport}
in the context of OT. In particular, we will consider maps $T_\U$ that are obtained 
as a sub-component of a triangular transport map by solving an appropriately 
constrained OT problem. Following \cite[Sec. 10.10(vii)]{bogachev2} and \cite{MGAN} we define the class 
of triangular maps:

\begin{definition}\label{def:triangular-map}
Consider  Hilbert spaces $\V, \Z, \U, \Y$.
A map $T: \Z \times \V \to \Y \times \U$
is (block) {\it triangular} if there exist maps $T_\Y: \Z \to \Y$ (the $\Y$-component)
and $T_\U: \Y \times \V \to \U$ (the $\U$-component) so that 
 $T(z,v) = \left( T_\Y(z), T_\U( T_\Y(z), v) )  \right)$ for all $ (z, v) \in \Z \times \V$.
  Henceforth we write $\TT( \Z, \V; \Y, \U)$ to denote the space of such maps.
\end{definition}

The significance of the triangular maps lies in the implication of
our  \Cref{prop:triangular-transport-from-MGAN}: 
Consider a  reference measure $\RefM \in \PP(\Y \times \V)$ of the form 
$\RefM = \RefM_\Z \otimes \RefM_\V$, where $\RefM_\Z \in \PP(\Z)$ and $\RefM_\V \in \PP(\V)$ are arbitrary, 
i.e., $\RefM$ is the independent coupling of $\RefM_\Z$ and $\RefM_\V$.
Then  if $T \in \TT(\Z, \V; \Y, \U)$ 
satisfied $T \# \eta = \nu$, its $\U$-component  $T_\U$ solves \Cref{prob:conditional-transport}. This observation leads to 
numerical algorithms that are capable of generating conditional samples 
for arbitrary choices of the conditioning variable $y$ as depicted in \Cref{fig:pinwheel-intro} where used our \Cref{alg:plugin}.

\begin{figure}[htp]
    \centering
    \begin{overpic}[width=.9 \textwidth, clip = true, trim = 3cm 0cm 3cm 0cm]{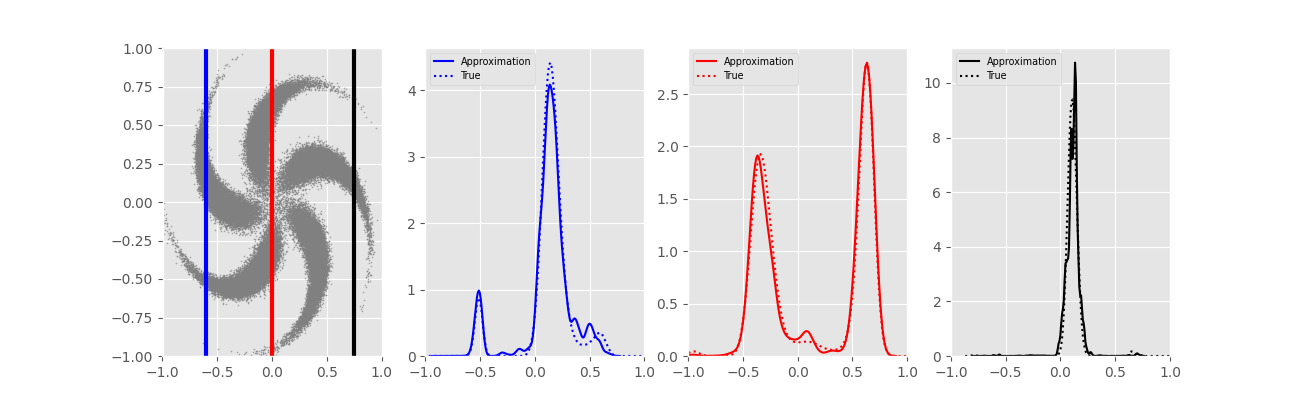}
        \put(-2, 18){\footnotesize $u$}
        \put(14, 0){\footnotesize $y$}
        \put(38.5, 0){\footnotesize $u$}
        \put(63.5, 0){\footnotesize $u$}
        \put(88, 0){\footnotesize $u$}
    \end{overpic}
    \caption{A depiction of conditional sampling using triangular maps obtained 
    from \Cref{alg:plugin}. The left panel shows a scatter plot of 
    samples from $\TargetM$ that were used to solve the conditional OT problem 
    while the blue, red, and black lines denote slices along which conditional 
    samples are generated. The other panels show the conditional histograms 
    with dotted lines showing the true histograms and solid lines showing 
    the numerical approximations.    }
    \label{fig:pinwheel-intro}
\end{figure}
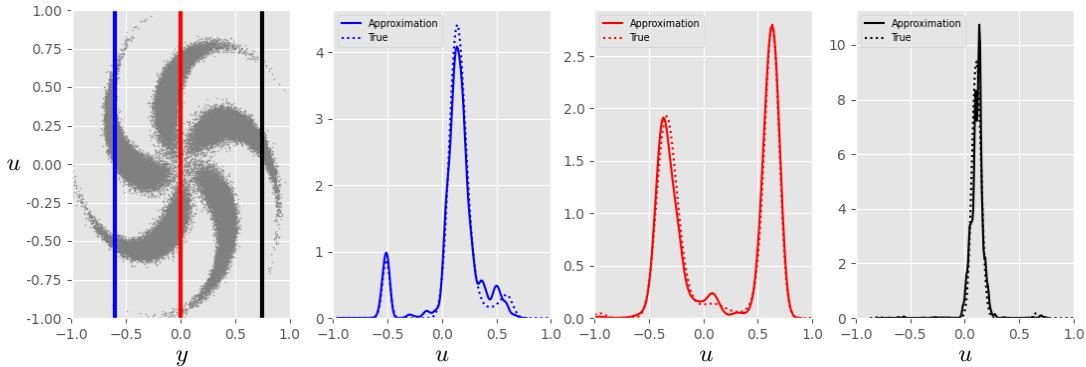

Since in most applications we have freedom in the choice of the reference $\eta$ 
we will present most of our results in the particular setting where $\Z = \Y$ and 
$\eta = \nu_\Y \otimes \eta_\V$ where $\nu_\Y$ denotes the $\Y$-marginal of $\nu$
and $\eta_\V$ is an arbitrary measure. This setup is natural for inverse problems since 
it is often easy for us to simulate $\nu_\Y$. For example, consider the setting of 
the Darcyflow 
inverse problem described earlier. Then we can generate samples from $\nu$ by the following procedure:  
(1) draw $u^j \sim \mu$, (2) solve the PDE \eqref{Darcy-PDE}  using a numerical solver to obtain the solution 
$p(u^j)$, and then (3) simulate the measurements $y^j = (p(u^j)(x_1), \dots, p(u^j)(x_m)) + \epsilon^j$. This procedure 
generates samples $(y^j, u^j) \sim \nu$ where 
$y^j \sim \nu_\Y$.

To this end,  we 
define the following OT problem which we refer to as the 
{\it conditional Monge problem} in the parlance of \cite{bunne2022supervised}:
\begin{problem}[Conditional Monge]\label{prob:conditional-monge}
Consider $\eta = \nu_\Y \otimes \eta_\V$ and 
a cost function $c: (\Y \times \V) \times (\Y \times \U) \mapsto \R$.  
    Find   $T^\star$ that solves 
    \begin{equation}\label{conditional-Monge}
        \left\{ 
\begin{aligned}
         & \inf_{T} M(T), &&M(T):= 
        \int_{\Y \times \V} c \big(z, v; T(z, v) \big) \dd \RefM( z, v), \\
        & \text{subject to (s.t.)} && 
        T(z, v) = ( z, T_\U(z, v) ) \quad {\rm and} \quad  T \# \eta = \nu. 
\end{aligned}
        \right.
    \end{equation}
\end{problem}
Note that here we are taking $T_\Y = \Id$ since we chose our reference $\eta$ to have the same $\Y$-marginal 
as $\nu$ and so we do not need to transport any mass in that direction.
Since the maps in this problem are clearly triangular then 
naturally the resulting $\U$-component of $T^\star$ will solve 
\Cref{prob:conditional-transport}, i.e., $T^\star_\U( y, \cdot) \# \eta_\V = \nu( \cdot \mid y)$. 

Letting $\Pi( \RefM ; \TargetM)$ denote the set of couplings between $\RefM$ and $\TargetM$
we then define the  {\it conditional Kantorovich problem} that serves as the 
relaxation of  the conditional Monge problem above:


\begin{problem}[Conditional Kantorovich]\label{prob:conditional-kantorovich}
Consider $\eta = \nu_\Y \otimes \eta_\V$ and 
a cost function $c: (\Y \times \V) \times (\Y \times \U) \mapsto \R$. 
    Find a coupling $\pi^\star$ that solves
    \begin{equation}\label{conditional-Kantorovich}
    \left\{
        \begin{aligned}
        & \inf_\pi K(\pi) && K(\pi):= \int_{\Y \times \V \times \Y \times \U} c(z, v; y, u)  \dd \pi(  z,  v;  y,  u)\\ 
        & \st && \pi \in \Pi_\Y( \eta, \nu) :=  \{  \gamma \in \Pi( \RefM ; \TargetM) \mid 
        (z, v; y, u) \sim \gamma, \quad \text{such that} \quad z = y \}.
        \end{aligned}
    \right.
    \end{equation}
\end{problem}
Simply put, this Kantorovich problem is solved over a constrained set of couplings where 
the $z, y$ variables are deterministically coupled by the identity map.

With this setup, our first result establishes the solvability of 
the conditional Kantorovich problem under mild assumptions and characterizes 
its optimal solution in terms of optimal couplings between $\eta_\V$ and 
the conditionals $\nu(\cdot \mid y)$. Indeed, our result states that 
 \Cref{prob:conditional-kantorovich} has an optimal solution that can be obtained 
 by coupling $z, y$ using the identity map and then coupling $\eta_\V$ 
 and $\nu(\cdot \mid y)$ optimally for a.e. $y$.
\begin{theorem}\label{main-thm:conditional-Kantorovich-solvability}
Consider the setting of \Cref{prob:conditional-kantorovich}, suppose that the cost 
$c$ is continuous, $\inf c > - \infty$, and that the infimum in \eqref{conditional-Kantorovich}
is finite
Then it holds that
\begin{enumerate}[label=(\roman*)]
    \item Problem \eqref{conditional-Kantorovich} admits a minimizer 
    $\pi^\star \in \PP( \Y \times \V \times \Y \times \U)$. 
    \item It holds that
    $\pi^\star(z, v, u \mid y) = \pi^\star (v, u \mid y) \delta( y- z)$
     where for $\nu_\Y$-a.e. $y \in \Y$  the conditional measures 
     $\pi^\star (v, u \mid y)$ are optimal for 
     \begin{equation*}
         \inf_{\pi \in \Pi( \eta_\V, \nu( \cdot \mid y))} 
         \int_{\V \times \U} c(y, v; y, u) \dd \pi(  v, u).
     \end{equation*}
\end{enumerate}
\end{theorem}
This theorem follows directly from \Cref{prop:conditional_kantorovich}. We note that 
the assumption that the infimum of \eqref{conditional-Kantorovich} is 
finite is often mild in practice and can be reduced to appropriate moment conditions 
on $\eta$ and $\nu$; see \Cref{rem:Finite-optimal-cost} for details.

Our second contribution is a strong duality result under slightly 
different assumptions on the cost $c$ which 
follows from \Cref{prop:dualValue}.

\begin{theorem}\label{main-thm:conditional-duality}
Consider \Cref{prob:conditional-kantorovich} with a lower semi-continuous and non-negative 
cost $c$. Then it holds that
    \begin{equation*}
       \begin{aligned}
                  \inf_{\pi \in \Pi_\Y(\eta, \nu)}  K(\pi) 
&= \sup_{\psi\in L^{1}(\RefM)}\left(\int_{\Y\times\U}\psi^c_v(z,u)\dd\TargetM(z,u)-\int_{\Y\times\V}\psi(z,v)\dd\RefM(z,v)\right), \\ 
&= \sup_{\phi\in L^{1}(\TargetM)}\left(\int_{\Y\times\U}\phi(y,u)\dd\TargetM(y,u)-\int_{\Y\times\V}\phi^c_u(y,v)\dd\RefM(y,v)\right),
       \end{aligned}       
    \end{equation*}
    where  
    $\psi^c_v(z, u) := \inf_v \bigg( \psi(z, v) + c(z, v; z, u)  \bigg)$ and 
    $\phi^c_u(y, v) := \sup_u \bigg( \phi(y, u) - c(y, v; y, u)  \bigg)$ are 
    the partial $c$-transfroms of $\psi$ and $\phi$ respectively.
\end{theorem}

For our third contribution we 
turn our attention to \Cref{prob:conditional-monge} and show that 
the conditional Kantorovich problem has a solution that is given by 
a unique transport map. This result is stated in \Cref{prop:conditional_monge_solvability} under slightly more general assumptions. 

\begin{theorem}\label{main-theorem:conditional-monge}
    Suppose $\V = \U$ and consider 
    the cost $c(z, v; y, u) = \| u - v \|_\U^p$ for $p \in (1, \infty)$. 
    Suppose that (a) the measures $\eta, \nu$ have finite $p$-th moments; (b) 
    $\eta_\V$ is regular according to \Cref{def:regular-measure};
    and (c) the infimum in \eqref{conditional-Kantorovich} is feasible and finite. 
    Then there exists a map $T^\star$ that uniquely solves \Cref{prob:conditional-monge}
    and  $\pi^\star = (\Id \times T^\star) \# \eta$ is 
    the unique solution to \Cref{prob:conditional-kantorovich}.
\end{theorem}

We can formally think of the map $T^\star$ as the map that is obtained by considering 
the OT maps $T^\star_y 
= \argmin_{T  \: : \: T\# \eta_\V = \nu(\cdot \mid y)} 
\int_\U \| u -T(u) \|_\U^p \dd \eta_\V(u)$ that transport $\eta_\V$
to the conditionals $\nu(\cdot \mid y)$ for a.e. $y$. One can then 
concatenate these maps along the $y$ coordinate to define the map $T_\U(y, v) = T_y(v)$
in \Cref{prob:conditional-monge}. This turns out to be 
the correct approach although it 
leads to delicate issues of measurable selection for $T_\U$ as dealt with in \cite{carlier2016vector}. Here we circumvent these technical details to the 
assumption that \eqref{conditional-Kantorovich} has a feasible solution which transfers 
the issue of measurable selection to the Kantorovich problem where it can be 
overcome easily.

Our next theoretical contribution concerns
the particular case of BIPs where we show that, 
in the finite dimensional case where $\U = \R^d$ and
under 
regularity assumptions on the likelihood $\Phi$ and the prior measure $\mu$, 
 the 
conditioning maps $T^\star_\U$ (the $\U$-components of the Monge solution) 
are  regular in $y$ and  further obtain error bounds for 
numerical approximations of these maps;
the detailed statement of these results can be found in 
\Cref{prop:prior-to-posterior-Brenier-map-with-rates}. 

Finally, we present a set of numerical experiments that explore various 
practical strategies for the solution of \Cref{prob:conditional-monge} and 
approximation of the maps $T_\U^\star(y, u)$ towards conditional simulation of 
probability measures. 
Our algorithms and experiments, including an instance 
of the Darcy flow inverse problem \eqref{Darcy-PDE} are collected in \Cref{sec:numerics}.

\subsection{Review of relevant literature}\label{subsec:rev:lit}
Below we review the relevant literature to our work with a particular 
focus towards triangular transport and conditional simulation. For a comprehensive review of OT theory see \cite{villani-OT,figalli2021invitation, ambrosio2005gradient, santambrogio2015optimal}
and for computational aspects  see \cite{peyre2019computational}.

\subsubsection{Triangular transport and OT}\label{subsec:rev:TT-OT}

The OT characterization of triangular transport maps 
goes back to  \cite{carlier2010knothe, bonnotte2013knothe}
that characterized the classic triangular rearrangements of Knothe \cite{knothe1957contributions} and 
Rosenblatt \cite{rosenblatt1952remarks} as the limit of a sequence of 
OT problems with a singularly perturbed quadratic cost.
We show in \Cref{subsec:relaxedProblem} that a similar 
limiting argument is applicable in our broader setting as well.
The Knothe-Rosenblatt (KR) map  is a classic transport map 
originally defined in the finite-dimensional setting which is 
built by successively transporting one measure into another one 
by 1D transport of their conditionals using deterministic maps; 
see \cite[Sec.~2.3]{santambrogio2015optimal} for a detailed definition. 
The KR map is an example of a fully triangular map according 
to our \Cref{def:triangular-map} and it was further  
generalized to triangular maps defined on Hilbert spaces 
in \cite{bogachev2005triangular, bogachev2006nonlinear} where their abstract
theory was developed; also see \cite[Sec.~10.10(vii)]{bogachev2}.

The OT characterization of triangular maps towards conditional 
simulation is a contemporary topic and the main results in this direction 
were presented in the seminal work \cite{carlier2016vector}
where 
an analogue for Brenier's theorem was presented for the 
characterization of parametric families of transport maps $T_\U$
that would characterize the conditionals $\nu(\cdot \mid y)$ of a 
given target measure $\nu$, i.e., $T_\U( y, \cdot) \# \eta_\U = \nu(\cdot \mid y)$.  However, the main goal of the authors in that work was to introduce 
a novel characterization of quantiles for vector valued random variables 
and the theory was not tailored to conditional simulation or triangular 
transport. This extension was discussed in the recent 
works \cite{wang2023efficient, al2023optimal, hosseini-taghvaei-OTPF-2, alfonso2023generative, taghvaei2022optimal, bunne2022supervised, muzellec2019subspace} although these 
works were focused on finite-dimensional settings with 
quadratic cost functions and are mainly concerned with the 
design of algorithms. 
Our analysis is heavily inspired by these previous works and in particular  \cite{carlier2016vector},
but deviates from them in three directions: (a) we develop 
our theory on the Hilbert and Polish space settings which is crucial for 
PDE inverse problems and high-dimensiona inference; 
(b) we study OT problems for arbitrary costs and give refined 
results for classic choices of the cost such as the quadratic case; (c) we give a detailed characterization of the corresponding constrained Kantorovich 
problems that lead to triangular Monge maps.

 \subsubsection{Triangular transport and inference}\label{subsec:rev:TT-UQ}

 The algorithmic aspects of 
 triangular transport maps for conditional sampling 
 and  Bayesian inference were developed in~\cite{marzouk2016sampling,marzouk-opt-map}
that primarily relies on the KR constructions of triangular 
maps for sampling of target measures or preconditioning of MCMC 
algorithms while the extension to general block-triangular 
maps  was presented in \cite{MGAN, alfonso2023generative} with the former developing algorithms 
using monotone triangular transport maps and the latter focusing 
on OT formulations akin to ours with quadratic costs.
More recently these methods have 
been extended for amortized inference with applications in 
geoscience
 \cite{baldassari2023conditional, yin2023solving, siahkoohi2023reliable, siahkoohi2021preconditioned} and  
 filtering and data assimilation \cite{spantini2019coupling,
 ramgraber2022ensemble, ramgraber2022ensemble-2} with  \cite{taghvaei2022optimal, al2023optimal, al2023optimal, grange2023computational} focusing on OT formulations 
 that are particular instances of our formulation of the conditional Monge 
 problem. These works further characterize the map $T^\star_\U$ 
 via a dual formulation using the conditional analogue of Brenier's theorem 
 presented in \cite{carlier2016vector} which is also the basis 
 of the methodologies developed in \cite{wang2023efficient, bunne2022supervised, muzellec2019subspace}.

\subsubsection{Triangular transport and generative modeling}\label{subsec:rev:TT-ML}
Broadly speaking, modern generative models  \cite{ng2002discriminative, jebara2012machine} 
solve the problem of generating samples from a target measure as a  transport  problem 
with a deterministic, stochastic, or dynamic formulation. These models have 
become very popular in recent years thanks to the success of various families of 
neural net based techniques such as GANs \cite{goodfellow2016nips,goodfellow2016deep, goodfellow2020generative},
NFs \cite{rezende2015variational, kobyzev2020normalizing, papamakarios2019normalizing, onken2021ot}, and
DMS \cite{sohl2015deep, song2020improved, kingma2021variational, cao2022survey, albergo2023stochastic, baptista-function-space-score}. Indeed GANs and NFs provide deterministic 
maps while DMs give a dynamic formulation in the form of an stochastic differential equation with 
a score function (i.e. drift) that is parameterized by a network. Among these models, NFs are 
the only ones that rely on triangular transport maps but the conditioning 
properties of these maps is rarely utilized or acknowledged in the NF literature. Instead, the triangular 
structure is mainly utilized for building maps that are easily invertible and their 
Jacobian determinants can be computed efficiently. 

While traditional 
generative modeling does not focus on conditional simulation, the methodology can be generalized to 
conditional simulation; see
\cite{trippe2018conditional, winkler2019learning, lueckmann2019likelihood, wang2023efficient} for conditional NFs, 
\cite{mirza2014conditional, adler2018deep, liu2021wasserstein, MGAN} for conditional GANs,
and \cite{batzolis2021conditional, saharia2022image} for conditional DMs. Here one considers
pushing a reference $\RefM_\U$ to the conditionals $\nu(\cdot \mid y)$ 
with a map $T_\U(y, \cdot)$ that is parameterized by $y$ and in this light our 
setting can be viewed as the OT analogue of such conditional generative models. 
The works \cite{adler2018deep, ray2022efficacy, zhou2022deep, 
baptista2020adaptive, wang2023efficient, MGAN} are perhaps the closest to our work 
in this domain although they do  not present a systematic study of conditional OT at our
level of generality. 

\subsubsection{Other interesting connections}\label{subsubsec:rev:interesting}

Here we mention a number of article in the literature that are related to our work 
but do not directly study the conditional OT or conditional simulation problems. 
While our focus here is not on algorithm development, we mention that the applied 
analysis of transport problems has become a vibrant area of research in recent years. 
The statistical analysis of OT maps has been an intense research area recently
\cite{hutter2021minimax, divol2022optimal, ghosal2022multivariate, deb2021rates, manole2021plugin} 
although the majority of these works, with the exception of \cite{hutter2021minimax}, focus on sample complexities and generalization 
bounds for transport maps as opposed to approximation errors due to parameterization of the 
maps. In contrast the article \cite{baptista2023approximation} gives a general framework for  
analyzing such approximation errors for minimum divergence estimators but these are different from OT maps. In parallel with the above a growing literature has been developed around the
analysis of triangular transport maps and in particular the KR map: 
 \cite{zech2022sparse-I, zech2022sparse-II} show that the KR map is analytic under some 
 assumptions and can be approximated with neural nets or sparse polynomials; 
 \cite{irons2021triangular, wang2022minimax} consider minimum KL estimators of KR maps for sampling and density estimation  and 
 study their statistical consistency and sample complexities; Finally,
 \cite{jaini2020tails} analyzes the tail behavior of triangular maps revealing an intricate balance between
 the tails of the reference and target measures and the expressivity of Lipschitz triangular maps.

 We also mention the literature on {\it sliced Wasserstein distances} \cite{rabin2010geodesic, rabin2012wasserstein, bonneel2015sliced, mahey2023fast} and {\it sliced 
 measure transport} \cite{li2023measure, li2023approximation}. Broadly speaking, 
the sliced Wasserstein distance is computed by taking random 1D projections of 
the reference and the target, and averaging the 1D Wasserstein distance between 
the requisite projections. The resulting distance imposes a similar topology to the 
Wasserstein distance but it can be computed very efficiently in practice. The resulting 
formulation has a lot of similarities to our conditional Monge problem \eqref{conditional-Monge}
for costs that are independent of the $z$ variable. Another related concept is 
that of pairs of measures with
{\it overlapping marginals} that arise in distributionally robust optimization \cite{ruschendorf1991bounds, embrechts2010bounds, fan2023quantifying}, these are 
precisely pairs of measures whose marginals on a given subspace match, akin to our
assumption that the $\RefM_\Y = \TargetM_\Y$  in \Cref{prob:conditional-monge} and \Cref{prob:conditional-kantorovich}.

\subsection{Outline}\label{subsec:outline}
The rest of the article is organized as follows: we review some preliminaries 
and  useful technical results in \Cref{sec:prelims}. Our main 
theoretical results are presented in \Cref{sec:conditionalOT}
followed by their application  to the particular case of BIPs in \Cref{sec:application}.
Our algorithms and numerical experiments are collected in \Cref{sec:numerics}
followed by our conclusions in \Cref{sec:conclusion}.

\section{Preliminaries}\label{sec:prelims}

In this section we collect some notations, definitions, and  preliminary results that are used throughout the rest of the article. We collect some notation in \Cref{subsec:notation}.
\Cref{subsec:RegConditionalMeasures} reviews conditional measures and 
for the most part 
follows \cite[Sec.~10.4]{bogachev2}
while \Cref{subsec:OTreview} recalls 
 useful results from OT theory following \cite{villani-OT, ambrosio2005gradient}.

\subsection{Notation}\label{subsec:notation}
Throughout the article we primarily work with Polish spaces and 
separable Banach or Hilbert spaces but the reader may take all of our function 
spaces to be separable and Hilbert in a first reading. We also reserve 
upper case calligraphic letters to denote function spaces. For a Polish space 
$\Y$ we write $\B(\Y)$ to denote its Borel sigma algebra and $\PP(\Y)$ to 
denote the space of Borel probability measures on $\Y$. For $\varrho \in \PP(\Y)$ and 
a function $f:\Y \to \R$ we write $\int_\Y f(y) \: \dd \varrho(y)$ to denote the expectation 
or integral of $f$ with respect to $\varrho$ and sometimes use 
the shorthand notation $\int f \dd \varrho$ when no confusion may arise. 
We often work on product spaces 
and need to compute integrals with respect to marginals and
write $\int_\Y f(x, y) \varrho(x, \dd y)$ to highlight the variable with respect to 
which the partial integral is computed. For $\vrho \in \PP( \Y \times \U)$ we 
write $\vrho_\Y$ and $\vrho_\U$ to denote its $\Y$ and $\U$-marginals respectively.

\subsection{Conditional measures and triangular maps}\label{subsec:RegConditionalMeasures}
Let $\Y, \U$ be Polish spaces and  for a set 
$B \in \B( \Y \times \U)$ define the slices $B^y: = \{ u \mid (y, u) \in B \}$. 
We recall the following definition of (regular) conditional measures following \cite{bogachev2}:

 \begin{definition}\label{def:conditional-measure}
  A function  $(B, y) \mapsto \vrho(B \mid y)$ is a (regular) \emph{conditional measure} of 
  $\vrho \in \PP(\Y \times \U)$ if: (a) For every fixed $y \in \Y$ it holds that $\vrho( \cdot \mid y) \in \PP(\U)$;
  (b) For every fixed $B \in \B(\U)$ the function $y \mapsto \vrho(B \mid y)$ is
    measurable with respect to $\B(\Y)$ and  $\vrho_\Y$-integrable;
  and (c) For all $B \in \B(\Y \times \U)$ it holds that
      $\vrho(B) = \int_\Y \vrho(B^y \mid y ) \vrho_\Y(\dd y).$
\end{definition}
We further recall the following  
facts  from 
\cite[Lem.~10.4.3 and Cor.~10.4.10]{bogachev2}:

\begin{proposition}
\label{conditional_measure_prop}
Consider $\varrho \in \PP(\Y \times \U)$ with $\Y, \U$ Polish spaces. 
  Then the conditional measures 
  $\vrho( \cdot \mid y )$ for $\vrho$ exist and 
  they are essentially unique, 
  i.e., there exists
    a set $E \in \B(\Y)$ so that $\vrho_\Y(E) = 0$ and the $\vrho(\cdot \mid y)$ 
    are unique for all
    $y \in \Y \setminus E$.
\end{proposition}


We can further characterize conditional measures  via 
triangular transport maps. 
The proof of the following proposition is nearly identical to 
\cite[Thm.~2.4]{MGAN} and is therefore omitted. 

\begin{proposition}\label{prop:triangular-transport-from-MGAN}
    Let $\Z, \V, \Y, \U$ be Polish spaces and 
    consider $\RefM\in \PP(\Z \times \V)$ and
    $\TargetM \in \PP( \Y \times \U)$. If $T \# \RefM = \TargetM$ for some $T \in 
    \TT( \Z, \V ; \Y, \U)$ (recall \Cref{def:triangular-map}) 
    then $T_\U(z, \cdot) \# \RefM(\cdot \mid z) = \nu( \cdot \mid T_\Y(z))$, for
    $\RefM_\Y$-a.e $z$. 
\end{proposition}

\subsection{Basics of OT}\label{subsec:OTreview}
Here we collect some preliminary results on OT as they 
pertain to our analysis in \Cref{sec:conditionalOT,sec:application}
and refer the reader to \cite{villani-OT,ambrosio2005gradient,santambrogio2015optimal,figalli2021invitation, peyre2019computational} for a 
details.

Let $\X, \W$ be Polish spaces and consider Borel probability measures $\RefM \in \PP(\X)$
and $\TargetM \in \PP(\W)$. We define $\Pi(\RefM, \TargetM) := \{ \pi \in \PP( \X \times \W) \: : \: \pi_\X = \RefM, \: \pi_\W = \TargetM \}$ as the space of {\it couplings or transference 
plans}
between $\RefM$ and $\TargetM$. Given a cost $c: \X \times \W \to (- \infty, \infty]$\footnote{Throughout the 
article we only consider real valued costs.} 
we consider the Kantorovich problem, 
\begin{equation}\label{standard-Kantorovich}
    \inf_{\pi \in \Pi(\RefM, \TargetM) } K(\pi), \qquad K(\pi) = \int_{\X \times \W} c(x, w) 
   \dd \pi (x, w),
\end{equation}
and the corresponding Monge problem, 
\begin{equation}\label{standard-Monge}
    \inf_{\substack{ T \# \RefM = \TargetM }} M(T), \qquad M(T) = \int_\X c(x, T(x))  \dd \RefM( x).
\end{equation}
We say that a coupling $\pi^\star$ is optimal if it 
solves \eqref{standard-Kantorovich} and similarly say that a map 
 $T^\star$
is {\it optimal} if it solves \eqref{standard-Monge}.
The existence of optimal couplings for the Kantorovich problem
can be established under very general conditions: 

\begin{proposition}[{\cite[Thm.~4.1]{villani-OT}}] 
\label{prop:standard-Kantorovich-solvability}
Let $\X, \W$ be Polish spaces and let $c$ be a lower semi-continuous 
function such that $\inf c > - \infty$ \footnote{This assumption can be slightly 
generalized to $c(x,w) \ge a(x) + b(w)$ for upper semi-continuous functions 
$a\in L^1_\RefM(\X), b \in L^1_\TargetM(\W)$ as 
is done in classic texts.}. Then there exists a minimizer of \eqref{standard-Kantorovich}.
\end{proposition}

Next we recall a technical result on measurable selections of couplings 
that will help us in our analysis of 
conditional Kantorovich problems \eqref{conditional-Kantorovich}, allowing us to build optimal couplings between 
measures from optimal couplings between their conditionals.

\begin{proposition}[{\cite[Cor.~5.22]{villani-OT}}]\label{measurable_selection}
    Suppose $\X,\W$ are Polish spaces and $c$ is
a continuous cost function such that $\inf c>-\infty$. Let $\Omega$ be
a measurable space and let $\omega \mapsto \left(\RefM_{\omega},\TargetM_{\omega}\right)$
be a measurable function from $\Omega$ to $\PP(\X)\times \PP(\W)$. Then there exists 
a measurable selection $\omega\to\pi_{\omega}$ such that
$\pi_{\omega} \in \Pi(\RefM_{\omega}, \TargetM_\omega)$ is optimal, i.e. $\pi_\omega$ solves \eqref{standard-Kantorovich} with $\eta_\omega$ and $\nu_\omega$ as marginals,  for every $\omega \in \Omega$.
\end{proposition}

Let us now consider the Monge problem \eqref{standard-Monge} and recall a 
characterization of its solution in the setting where $\W = \X$ are separable Hilbert spaces. 
First, we recall the definition of a (Gaussian) regular measure.  

\begin{definition}\label{def:regular-measure}
    Suppose $\X$ is a separable Hilbert space. We say $B \in \B(\X)$ is a 
    Gaussian null set if $\gamma(B) =0$ for any non-degenerate Gaussian measure $\gamma \in \PP(\X)$ \footnote{Recall that $\gamma$ is said to be a Gaussian measure 
    if $\ell \# \gamma$ is a one-dimensional Gaussian distribution for all $\ell \in \X^\star$
    the dual of $\X$ \cite{bogachev-gaussian}.}.
    Furthermore, we say that a measure $\vrho \in \PP(\X)$ is regular if $\vrho(B) = 0$
    for all Gaussian null sets $B$.
\end{definition}

With the above notion of a regular measure at hand we can present the 
following result regarding the existence of unique optimal solutions to 
\eqref{standard-Monge} in relation to \eqref{standard-Kantorovich}.

\begin{proposition}[{\cite[Thm.~6.2.10]{ambrosio2005gradient}}]\label{ambrosio_hilbert_space}
Let $\X$ be a separable Hilbert space and 
consider the cost $c\left(u,v) \right)= \| u-v\|^p_\X$
for $p\in\left(1,\infty\right)$. Suppose  $\RefM, \TargetM \in \PP(\X)$ are both measures
with finite $p$-moments, and assume $\RefM$
is regular according to \Cref{def:regular-measure}. 
Then it holds that: 
\begin{enumerate}[label=(\roman*)]
    \item Problem \eqref{standard-Kantorovich} has a unique solution $\pi^\star \in \Pi(\RefM, \TargetM)$, 
    such that $\pi^\star=\left(I\times T^\star\right)\#\RefM$ for an optimal
    map $T^\star \in L_\eta^{p}(\X; \X)$ that solves \eqref{standard-Monge}.

    \item If $\nu$ is also regular then $T^\star$ is  injective $\eta$-a.e. 


\end{enumerate}
\end{proposition}
Here $L^p_\eta( \X; \X)$ denotes the $\eta$-weighted 
$L^p$ space of Bochner integrable maps from $\X$ to itself. Finally we present a 
technical lemma about the convergence of 
OT maps in $L^2$ spaces which may be known by experts but we present the proof for convenience.
\begin{lemma}\label{lem:mapconvergence}
    Let $\X$ be a separable Hilbert space, and let $\RefM,\TargetM \subset \PP(\X)$ have finite second moments. Let $\pi_k = (\Id \times T_k)\# \RefM$ be a sequence of deterministic couplings between $\RefM$ and $\TargetM$, defined by maps $T_k$.
     If the couplings $\pi_k$ converge weakly, that is,  $\pi_k \rightharpoonup \pi$, where $\pi = (\Id \times T) \# \RefM$, then the maps $T_k \to T$ in $L^2_\RefM(\X;\X)$.
\end{lemma}
\begin{proof}
    Note that $\|T_k\|^2_{L^2_{\RefM}(\X; \X)} = 
    \|T\|^2_{L^2_{\RefM}(\X; \X)} = \int_\X\|y\|^2\TargetM(\dd y)$ for all $k$
    where the second equality follows from the fact that $T\#\RefM = \TargetM$ and 
    similarly for $T_k$. 
    Since $\pi_k\rightharpoonup\pi$ then  
    $\lim_{k\to\infty} \int \phi \: \dd\pi_k = \int \phi \: \dd\pi$
    for all continuous functions $\phi:\X\times\X\to \mathbb{R}$. We then have that 
    \[
    \int_{\X\times\X}\phi(x,x')\dd \pi_k( x,  x') = 
    \int_{\X\times\X}\phi(x,x') \dd (\Id \times T_k)\#\eta( x, x')
    =\int_{\X\times\X}\phi(x,T_k(x)) \dd \RefM(x).
    \]
    Repeating a similar calculation for $\pi$ yields that 
    \[
    \lim_{k\to\infty} \int_{\X\times\X}\phi(x,T_k(x)) \dd \eta( x) = 
    \int_{\X\times\X}\phi(x,T(x)) \dd \eta(x).
    \]
    Let $\phi(x,y)=\langle\psi(x),y\rangle$ for a continuous function $\psi:\X\to\X$. Then we see that
    $\lim_{k\to\infty} \int_X\langle\psi(x),T_k(x)\rangle\dd\eta(x) = 
    \int_X\langle\psi(x),T(x)\rangle\dd\eta(x)$.
    Thus, since the continuous functions on $\X$ are  dense in $L^2_\RefM(\X;\X)$
    \footnote{As a consequence of the density of simple functions in $L^2_\RefM(\X; \X)$ and the fact that simple functions on separable Hilbert spaces can be approximated by continuous 
    functions by Urysohn's theorem \cite[Page 7]{Yosida1965}.}
    , and the sequence 
    $(T_k)$ is bounded, the equality above is extended to all functions $\psi\in L^2_\RefM(\X;\X)$, and we see that $T_k \rightharpoonup T$. Finally, we apply the Radon-Riesz property in Hilbert spaces to conclude that weak convergence combined with equality (convergence) of the norms implies strong convergence in $L^2_\RefM(\X; \X)$. 
\end{proof}

\section{General theory of conditional OT}\label{sec:conditionalOT}

In this section we present our main theoretical results 
in the form of propositions and theorems that amount 
to  the extended
versions of \Cref{main-thm:conditional-Kantorovich-solvability,main-thm:conditional-duality,main-theorem:conditional-monge}. 
We study the solvability of the conditional Kantorovich problem 
in \Cref{subsec:conditionalKantorovich}; present a strong 
duality result in \Cref{subsec:conditionalDual}; and investigate the 
solvability of the conditional Monge problem in \Cref{subsec:conditionalMonge}. 
We further  characterize the 
conditional Kantorovich problem as the limit of  unconstrained  problems with a perturbed cost in 
\Cref{subsec:relaxedProblem} followed by an extension  
to  broader classes of triangular maps in 
\Cref{subsec:generic-triangular-OT}.

Throughout this section (except for \Cref{subsec:generic-triangular-OT}) 
we will consider a reference measure $\RefM \in \PP( \Y \times \V)$ 
and a target $\TargetM \in \PP( \Y \times \U)$ for 
function spaces $\Y, \V, \U$ which are either taken to be 
Polish or separable Hilbert spaces. We further make the following 
assumption regarding the $\Y$-marginals of these two measures. 
\begin{assumption}\label{assump:y-marginals-match}
The reference $\eta \in \PP(\Y \times \V)$ and target $\nu \in \PP(\Y \times \U)$ 
satisfy $\RefM_\Y = \TargetM_\Y$.
\end{assumption}
Under this assumption we can take the map $T_\Y = \Id$ in \Cref{def:triangular-map} and work with the constrained set of 
couplings $\Pi_\Y$ from \eqref{conditional-Kantorovich}
which we recall for convenience:
\begin{equation}\label{conditional-coupling-Id}
    \Pi_\Y(\RefM; \TargetM) = \left\{ \pi \in \Pi(\RefM; \TargetM) \: :  \: 
    (z, v; y, u) \sim \pi \text{  such that  } y = z 
    \right\}\footnote{Note that this set is empty whenever \Cref{assump:y-marginals-match} is violated.}.
\end{equation}
By definition, any  $\pi\in\Pi_{\Y}\left(\RefM; \TargetM\right)$
 can be written as 
\begin{equation}\label{eq:pi-expansion}
\pi(z,v,y,u)
= \pi(v,u \mid y, z) \pi(y \mid z) \pi(z) 
= \widetilde{\pi}(v,u \mid y)\delta(y-z)\RefM_{\Y}(z),
\end{equation}
as these couplings are 
concentrated on the set $\{y=z\}$. 
We now verify that the conditional measure  
$\wt{\pi}(v,u \mid y)$ is a coupling  of  
$\RefM( \cdot \mid y)$ and $\TargetM( \cdot \mid y)$ a. e. 
as a particular application of the following results by taking $T_\Y = \Id$.

\begin{proposition}\label{prop:Conditional Coupling Conditioning}
Let $T_\Y: \Y \to \Y$ be an invertible map
such that $T_\Y \# \eta_\Y = \nu_\Y$ and consider the class of couplings 
\begin{equation*}
    \Pi( \RefM, \TargetM; T_\Y):= \{ \pi \in \Pi(\RefM, \TargetM) : 
    (z, v; y, u) \sim \pi \text{ such that } y = T_\Y(z)  \} .
\end{equation*}
Then every coupling $\pi \in \Pi(\RefM, \TargetM; T_\Y)$ can be written as 
$\pi(z, v, y, u) = \wt{\pi}( v,u \mid y) \delta( y - T_\Y(z)) \RefM_\Y(z)$
where  $\widetilde{\pi}(v,u\mid y) \in \Pi \left( \RefM(v\mid T_{\Y}^{-1}(y)),  \TargetM(u\mid y)  \right)$ for  $\nu_{\Y}$-a.e.  $y$. 
\end{proposition}

\begin{proof}
The expression of $\pi$ in terms of  $\wt{\pi}$ follows from the expansion $\pi(z,v,y,u) = \pi(v,u|y,z) \pi(y|z)\pi(z)$ and the constraint that $y=T_\Y(z)$ almost surely with respect to $\pi$. In order to show the $\V$ marginal of $\widetilde{\pi}(v,u\mid y)$ is $\RefM(v\mid T_{\Y}^{-1}(y))$, note that 
\begin{equation*}
\int_\Y \int_\U \pi(z,v, \dd y,\dd u)=\RefM(z,v)= \RefM(v \mid z) \RefM_\Y(z),
\end{equation*}
because $\pi$ is a coupling of $\RefM$ and $\TargetM$. Moreover, the expansion of $\pi$ in terms of $\wt{\pi}$ leads to 
\begin{equation*}
\begin{aligned}
\int_\Y \int_\U \pi(z,v, \dd y,\dd u)
& = \int_\Y \int_\U  \wt{\pi}( v, \dd u \mid y) \delta(\dd y - T_{\Y}(z)) \eta_\Y (z) =  \eta_\Y (z)  \int_\U  \wt{\pi}( v, \dd u \mid T_\Y(z)).
\end{aligned}
\end{equation*}
Hence, we have shown that 
$\int_{\U}\wt{\pi}(v, \dd u \mid T_\Y(z))=\RefM(v\mid  z)$ for a.e. $z$. Replacing $z$ with $T_\Y^{-1}(y)$ yields the conclusion
$    \int_{\U}\wt{\pi}(v, \dd u \mid y)=\RefM(v\mid  T_\Y^{-1}(y))$ for a.e. $y$. 
The proof is completed by following the same procedure to show that the $\V$ marginal of $\wt{\pi}(v,u\mid y)$ is $\TargetM(u\mid y)$. In particular, 
$\int_\Y \int_\V \pi(\dd z, \dd v, y, u)=\TargetM(y,u)=\TargetM(u \mid y)\TargetM_\Y(y)$
and
\begin{equation*}
\begin{aligned}
\int_\Y \int_\V \pi(\dd z, \dd v, y, u)
& = \int_\Y \int_\V  \wt{\pi}( \dd v,  u \mid y) \delta(y - {T_\Y( z)}) \eta_\Y (\dd z) \\
&= \int_\Y \int_\V  \wt{\pi}( \dd v,  u \mid y) \delta(y -w) T_\Y \#\eta_\Y (\dd w) \\
&= \int_\Y \int_\V  \wt{\pi}( \dd v,  u \mid y) \delta(y -w) \nu_\Y (\dd w) =   \nu_\Y (y) \int_\U  \wt{\pi}( \dd v,  u \mid y)   
\end{aligned}
\end{equation*}
concluding
$    \int_{\V}\wt{\pi}(\dd v, u \mid y)=\TargetM(u\mid y)$ for a.e. $y$.
{}
\end{proof}

We interpret the above result as 
the coupling analog of 
 \Cref{prop:triangular-transport-from-MGAN}, stating that the constraint $y = T_\Y(z)$
automatically enforces $\pi \in \Pi_\Y$ to couple the conditionals 
$\RefM( \cdot \mid z)$ and $\TargetM(\cdot \mid y)$ in an analogous manner 
in which the $\U$-component $T_\U$ of a triangular transport map $T$
that pushes $\RefM$ to $\TargetM$ will satisfy $T_\U(y , \cdot ) \# 
\RefM(\cdot \mid T_\Y^{-1}(y)) = \TargetM( \cdot \mid y)$. 
For most of our results in this section we apply \Cref{prop:Conditional Coupling Conditioning} with $T_\Y = \Id$ under \Cref{assump:y-marginals-match} but 
later use this result to obtain more general triangular OT maps.

\subsection{Solvability of the conditional Kantorovich Problem}\label{subsec:conditionalKantorovich}
We now consider the  conditional Kantorovich problem \eqref{conditional-Kantorovich}. Let us  recall this problem for convenience:
\begin{equation}\label{eq:Constrained Kantorovich}
\inf_{\pi\in\Pi_{\Y}  \left(\RefM,\TargetM\right)} K(\pi).
\end{equation}
It is also helpful for us to equivalently write this problem 
as an unconstrained Kantorovich problem with a modified cost function:
\begin{equation}\label{eq:kantorovich-unconstrained}
\left\{
\begin{aligned}
& \inf_{\pi\in\Pi\left(\RefM,\TargetM\right)}  \int c_\chi(z, v; y, u) \: \dd \pi( z, v, y, u), \\ 
& c_\chi(z,v ; y, u):= c(z, v; y, u) + \chi(y,z), \qquad 
     \chi(y,z) := 
 \begin{cases}
0 & y=z,\\
+\infty & \text{otherwise}.
\end{cases}
\end{aligned}
\right.
\end{equation}
The minimizers and minimum of these problems coincide, and the two
problems are completely equivalent under \cref{assump:y-marginals-match}. 
We are now ready to present our main result concerning the solvability 
of \eqref{eq:Constrained Kantorovich} and the characterization 
of its minimizers as optimal couplings between the conditionals $\RefM(\cdot \mid y)$ and $\TargetM(\cdot \mid y)$. Our main  \Cref{main-thm:conditional-Kantorovich-solvability} also follows from this result.

\begin{proposition}\label{prop:conditional_kantorovich}
 Let $\V, \U, \Y$ be Polish spaces and consider  a cost function 
 $c: (\Y \times \V) \times (\Y \times \U) \to \R$ that is 
 continuous and $\inf c > - \infty$.
Suppose \Cref{assump:y-marginals-match} holds and that 
the infimum in \eqref{eq:Constrained Kantorovich}
 is finite.
 Then it holds that: 

 \begin{enumerate}[label=(\roman*)]
     \item  Problem \eqref{eq:Constrained Kantorovich} admits a minimizer $\pi^\star \in \PP( \Y \times \V \times \Y \times \U)$.

\item  $\pi^\star\left(z,v,u \mid y\right) =\widetilde{\pi}^\star(v,u \mid y) \delta(y - z)$ where 
for $\eta_\Y$ a.e. $y$, $\widetilde{\pi}^\star(v, u \mid y)$
is optimal for 
\begin{equation}\label{pi-y-def}
\inf_{\pi \in\Pi\left(\RefM\left(v\mid y\right),\TargetM(u\mid y)\right)}
\int_{\V \times \U} c^y(v,u) \: \dd \pi \left( v,  u\right), \qquad 
c^y(v,u) = c(y, v; y, u).
\end{equation}
 \end{enumerate}
 

\end{proposition}

\begin{proof}
Our proof strategy is to choose a candidate for  $\pi^\star$ which  is optimal for \eqref{pi-y-def}. Using this candidate we then define $\pi^\star$ and verify claims (i, ii).
By \Cref{prop:standard-Kantorovich-solvability}, the problem \eqref{pi-y-def} has a solution $\pi^y$ for every $y$. 
Furthermore, 
 the map
$y\to\left(\RefM(\cdot\mid y),\TargetM(\cdot\mid y)\right)$
is measurable by  \Cref{conditional_measure_prop} and so by \Cref{measurable_selection}
 there exists a measurable choice $y\mapsto\pi^{y}$ such that $\pi^{y}$ is an
optimal transference plan between  $\RefM\left(\cdot\mid y\right)$ and
$\TargetM\left(\cdot\mid y\right)$ \footnote{Note that the 
measurable selection argument is crucial here since the Kantorovich problems 
between the conditional measures $\RefM(\cdot \mid y)$ and $\TargetM(\cdot \mid y)$ can have multiple minimizers and so the mapping $y \to \pi^y$ may be multi-valued.}. 
We can therefore define the candidate coupling
$\pi^{\star}(z,v,y,u)=\pi^{y}(v,u)\delta(y-z)\RefM_{\Y}(y)$
which we will now show is optimal. 
By construction, $\pi^{\star}\in\Pi_{\Y}\left(\RefM,\TargetM\right)$
and $\wt{\pi}^{\star}(v,u\mid y) = \pi^y(v,u)$ which solves \eqref{pi-y-def}.
We now have, by definition of $c^y$, that
\begin{equation*}
\begin{aligned}
 \int_{\Y}\Big[\int_{\mathcal{V\times\U}}c^y\left(v,u\right)& \dd\pi^{y}(v,u)\Big]\dd\RefM_{\Y}(y)
= \\
& \int_{\Y\times\Y}\left[\int_{\mathcal{V\times\U}}c\left(z,v; y, u\right)\dd\pi^{y}(v,u)\right]\dd\delta(y-z)\dd \RefM_{\Y}(y) 
= K(\pi^{\star}).
\end{aligned}
\end{equation*}
Now consider an arbitrary  $\pi \in \Pi_\Y(\RefM, \TargetM)$ 
as in~\eqref{eq:pi-expansion}, with associated cost $K(\pi)$, towards showing
that $K(\pi)\geq K(\pi^\star)$, which would imply that $\pi^{\star}$ is optimal. 
Disintegration along $y,z$ fibers, and recalling the notation of \Cref{prop:Conditional Coupling Conditioning}, gives
\begin{equation*}
\begin{aligned}
K(\pi)& =\int_{\Y\times\Y}\left[\int_{\mathcal{V\times\U}}c\left(z, v; y, u\right)\dd\pi\left(v,u\mid y,z\right)\right]\dd\pi(y,z), \\
& = \int_{\Y}\left[\int_{\mathcal{V\times\U}}c^y\left(v,u\right)\dd \wt{\pi}\left(v,u\mid y\right)\right]\dd\RefM_{\Y}(y),
\end{aligned}
\end{equation*}
where the second identity follows because  $\pi$ is concentrated
 on $\{ y=z \}$ and its $\Y$-marginal is $\RefM_{\Y}$ since it is an element of $\Pi_\Y(\RefM, \TargetM)$. 
We can now compute
\begin{align*}
K(\pi)& -K(\pi^{\star})  \\
&= \int_{\Y}\left[\int_{\mathcal{V\times\U}}c^y\left(v,u\right)\dd\wt{\pi}(v,u \mid y)\right]\dd\RefM_{\Y}(y)-\int_{\Y}\left[\int_{\mathcal{\mathcal{V\times\U}}}c^y\left(v,u\right)\dd\pi^y\left(v,u\right)\right]\dd\RefM_{\Y}(y), \\
&= \int_{\Y}\left[\int_{\mathcal{\mathcal{V\times\U}}}c^y\left(v,u\right)\dd\wt{\pi}\left(v,u\mid y\right)-\int_{\mathcal{\mathcal{V\times\U}}}c^y\left(v,u\right)\dd\pi^{y}(v,u)\right]\dd\RefM_{\Y}(y) \ge 0,
\end{align*}
where the last inequality follows from the fact that
$\pi^{y} \in \Pi( \RefM(v\mid y) , \TargetM(u\mid y))$ is optimal 
for each $y$ and so $\int_{\mathcal{\mathcal{V\times\U}}}c^y\left(v,u\right)\dd\wt{\pi}\left(v,u\mid y\right)-\int_{\mathcal{\mathcal{V\times\U}}}c^y\left(v,u\right)\dd\pi^{y}(v,u) \ge 0$. This establishes the optimality of $\pi^\star$. We can further take $\wt{\pi}^\star 
= \pi^y$ to conclude the proof of (i), and (ii) is satisfied by construction.

\begin{remark}\label{rem:uniqueness}
The optimality of $\widetilde{\pi}^\star(v, u \mid y)$ in \cref{pi-y-def} in combination with the essential uniqueness of the conditional measures from \Cref{conditional_measure_prop} means that whenever the cost function $c^y(v,u)$ and the conditional measures $\RefM(v\mid z)$ satisfy conditions for uniqueness of OT solutions for all $y$, the conditional OT coupling is also unique up to measure zero sets. In particular, this holds in the case that $c^y(v,u) = c(v,u) =  \|u-v\|_\U^{p}$, or more generally, $c(v,u)=h(\|u-v\|_\U)$ for strictly convex  $h$ and under the assumption that $\RefM(v\mid z)$ is regular.
\end{remark}

{}
\end{proof}

\begin{remark}\label{rem:Finite-optimal-cost}
We note that \Cref{prop:conditional_kantorovich} 
assumes the existence of a  finite cost 
coupling for \eqref{eq:Constrained Kantorovich}. While this assumption 
may seem stringent at first sight, it is often easy to verify in practice.
For example, consider the case where 
$\U$ is a separable Hilbert space, and $\mathcal{V=U}$ with 
$c(z, v; y, u)=\|v-u\|_\U^{2}$, then a simple assumption that
ensures the existence of a 
finite coupling is that $\int\|u\|_\U^{2}d\RefM<\infty$ and $\int\|v\|_\U^{2}d\TargetM<\infty$
\footnote{In fact, this is precisely Assumption (M) in  \cite{carlier2016vector}.}.
To see this, consider the coupling
$\pi(z,v,y,u)=\RefM\left(v\mid y\right)\TargetM\left(u\mid y\right) \delta(y-z) \RefM_{\Y}(y)$. 
 We call this the conditional independence coupling since it is the law of
the random variable $(z, v, y, u)$ by drawing $y\sim\RefM_{\Y}$,
setting $z=y$ and then sampling $v\sim\RefM\left(\cdot \mid y\right)$ and $u\sim\TargetM\left(\cdot \mid y\right)$
independently. We can now write, recalling that $\RefM_\Y = \TargetM_\Y$,
\begin{align*}
 K(\pi^{\star}) &\leq  \int_{\mathcal{Y\times U\times Z\times V}} \|v-u\|_\U^{2} \dd\pi(z, v, y ,u) \\
&\leq  \int_{\mathcal{Y\times U\times \Y\times V}} \left(\|v\|_\U^{2}+\|u\|_\U^{2}\right)   \RefM(\dd v\mid y)\TargetM(\dd u\mid y) \dd \delta(z-y) 
\dd \RefM_{\Y}(y). \\
&=\int_{\Y \times \U} \|v\|_\U^{2} \dd\RefM(z, v) 
+ \int_{\Y \times \U} \|u\|_\U^{2} \dd\TargetM(y, u) < \infty.
\end{align*}

We also note that the above argument can be extended to broader classes of costs that satisfy 
$c(z,v; y, u)\leq C \left(c(z, v; z,  v') + c(y, u; y, u')\right),$  for some constant $C\ge 0$ and fixed points $v', u' \in \V \times  \U$.
Then if $\int c(z, v; z, v')\dd\RefM(z, v)<\infty$ and  $\int c(y, u; y, u')\dd\TargetM(y, u)<\infty$ then the finite
cost assumption can be verified.
\end{remark}

\subsection{The conditional dual problem}\label{subsec:conditionalDual}
We now turn our attention to the dual problem of \eqref{eq:Constrained Kantorovich} 
and show that a strong duality result, akin to the standard result
\cite[Thm.~5.10]{villani-OT} holds; in fact, our proof relies heavily on the 
standard setting after considering the unconstrained 
formulation \eqref{eq:kantorovich-unconstrained}.


\begin{proposition}\label{prop:dualValue}
     Let $\V, \U, \Y$ be Polish spaces and suppose \Cref{assump:y-marginals-match} holds. 
 Consider a cost function $c:(\Y \times \V) \times (\Y \times \U) \to\mathbb{R}$ that is 
 lower semicontinuous and 
 $\inf c > - \infty$ 
 \footnote{A slightly more general assumption would the existence of upper semicontinuous functions $a \in L^1_\RefM(\Y \times \V)$ and $b\in L^1_\TargetM(\Y \times \U)$ such that 
$    c(z, v; y, u) \geq a(z,v)+b(y,u) $
for all $(z,v,y,u)$.}.  
Then it holds that
\begin{align*}
\min_{\pi\in\Pi_{\Y}(\RefM,\TargetM)} K(\pi) 
&= \sup_{\substack{(\psi,\phi)\in L^{1}(\RefM)\times L^{1}(\TargetM) \\ \phi-\psi\leq c_\chi }}\left(\int_{\mathcal{Y\times U}}\phi(y,u)\dd\TargetM(y,u)
-\int_{\mathcal{Y\times V}}\psi(z,v)\dd\RefM(z,v)\right), \\
&= \sup_{\psi\in L^{1}(\RefM)}\left(\int_{\Y\times\U}\psi^c_v(z,u)\dd\TargetM(z,u)-\int_{\Y\times\V}\psi(z,v)\dd\RefM(z,v)\right), \\
&= \sup_{\phi\in L^{1}(\TargetM)}\left(\int_{\Y\times\U}\phi(y,u)\dd\TargetM(y,u)-\int_{\Y\times\V}\phi^c_u(y,v)\dd\RefM(y,v)\right),
\end{align*}
where the partial $c$-transforms  of $\psi$ and $\phi$ are defined as
\begin{equation*}
    \psi^c_v(z, u) := \inf_v \bigg( \psi(z, v) + c(z, v; z, u)  \bigg), 
    \qquad 
    \phi^c_u(y, v) := \sup_u \bigg( \phi(y, u) - c(y, v; y, u)  \bigg).
\end{equation*}

\end{proposition}
\begin{proof}
Consider the formulation \eqref{eq:kantorovich-unconstrained}.
 Observe that this problem, with the cost $c_\chi$, satisfies the 
assumptions of \cite[Thm.~5.10]{villani-OT}
for strong duality and so
\begin{equation}\label{disp:standard-duality-c-chi}
\begin{aligned}
\min_{\pi\in\Pi_{\Y}(\RefM,\TargetM)} K(\pi) 
&= \sup_{\substack{(\psi,\phi)\in L^{1}(\RefM)\times L^{1}(\TargetM); \\ \phi-\psi\leq \widetilde{c}}}\left(\int_{\mathcal{Y\times U}}\phi(z,v)\dd\TargetM(y,u)-\int_{\mathcal{Y\times V}}\psi(y,u)\dd\RefM(z,v)\right), \\
&= \sup_{\psi\in L^{1}(\RefM)}\left(\int_{\Y\times\U}\psi^{c_\chi}(y,u)\dd\TargetM(y,u)-\int_{\Y\times\V}\psi(z,v)\dd\RefM(z,v)\right), \\ 
&= \sup_{\phi\in L^{1}(\TargetM)}\left(\int_{\Y\times\U}\phi(y,u)\dd\TargetM(y,u)
-\int_{\Y\times\V}\phi^{c_\chi}(z,v)\dd\RefM(z,v)\right),
\end{aligned}
\end{equation}
where $\psi^{c_\chi}(y,u):=\inf_{(z,v)}\left(\psi(z,v)+c_\chi(z,v;y,u)\right)$ and similarly,  $\phi^{c_\chi}(z,v):=\sup_{(y,u)}\left(\phi(y,u) - c_\chi(z,v;y,u)\right)$.
Since $c_\chi$ is real valued when restricted to $\{y=z\}$ and infinite
valued when $\{y\neq z\}$, we see that when taking the infimum in the definition of $\psi^{c_\chi}$, we may
restrict to  $\{y=z\}$ and so 
\begin{equation*}
\psi^{c_\chi}(y,u)=\inf_{z,v}\left(\psi(z,v)+ c_\chi(z,v;z,u)\right)
=\inf_{v}\left(\psi(z,u)+c(z, v; z, u)\right)
=\psi^c_v(z,v)
\end{equation*}
and a similar calculation can be repeated to show $\phi^{c_\chi} = \phi^c_u$.
Substituting back into \eqref{disp:standard-duality-c-chi} completes the proof.
\end{proof}

\begin{remark}
    In the case that $c$ is a metric we recover an analogous
    version of the Kantorovich-Rubinstein duality \cite[Particular Case 5.16]{villani-OT} since in 
    this case $c_\chi$ is also a metric 
    as the sum of the metric $c$ and an $\infty$-valued discrete metric on the $y$ and $z$ variables. Therefore,
\begin{equation*}
\min_{\pi\in\Pi_{\Y}\left(\RefM,\TargetM\right)} K(\pi)=\sup_{\psi\in L^{1}(\RefM) 
\cap \Lip^1_{c_\chi}}
\left(
\int \psi(z, v) \: \dd \eta(z, v) 
-
\int \psi(y, u) \: \dd \nu(y, u) 
\right),
\end{equation*}
where $\Lip^1_{c_\chi}$ is the space of 1-Lipschitz functions with respect to $c_\chi$.
\end{remark}


\subsection{Solvability of the conditional  Monge problem}\label{subsec:conditionalMonge}
We dedicate this section to establishing the solvability of the 
conditional Monge problem \eqref{conditional-Monge} which we 
recall below:
\begin{equation}\label{eq:constrained-Monge}
    \min_T M(T), \quad \st \quad T\# \eta =\nu, \text{ and } T(z, v) = (z, T_\U(z, v) ). 
\end{equation}
Our main result in this direction establishes the 
existence of unique solutions for this problem under 
structural assumptions on the cost as well as regularity assumptions on the 
reference
measure $\RefM$. We state the result for simple families of costs $c$
 that are often used in practice (see for example 
\cite{carlier2016vector, al2023optimal, taghvaei2022optimal, alfonso2023generative, hosseini-taghvaei-OTPF-2}) and provide a 
remark later on how the result can be generalized to broader choices.




\begin{proposition}\label{prop:conditional_monge_solvability}
Let $\U$ be a separable Hilbert space, take $\V = \U$, and let $\Y$
be a Polish space. Consider the cost 
$    c(z, v; y, u) = c(v,u) =    \|u-v\|_\U^{p}$
for $p\in\left(1,\infty\right)$.   
Suppose that (a) the measures $\RefM$ and $\TargetM$ satisfy 
\Cref{assump:y-marginals-match} and 
 have finite $p$-th moments and
 (b) the reference conditionals $\RefM( \cdot \mid y)$
 assign zero measure to Gaussian null sets, i.e., regular 
 according to \Cref{def:regular-measure}, for $\RefM_\Y$-a.e. $y$.
Then $\pi^\star = (\Id \times T^\star)\# \eta$ 
is  the unique
solution of \eqref{eq:Constrained Kantorovich}
where $T^\star$ uniquely solves \eqref{eq:constrained-Monge}.
\end{proposition}

\begin{proof}
Let $\pi^{\star}$ be a solution to problem \eqref{eq:Constrained Kantorovich}, which 
exists following \Cref{prop:conditional_kantorovich} 
and has finite cost by \Cref{rem:Finite-optimal-cost}. Disintegrating $K(\pi^\star)$ 
we can write
\begin{align*}
&\inf_{\pi} \int_{\Y\times\U\times\Y\times\U} c(v,u) \, \dd\pi(z,v,y,u) 
= \int_{\Y\times\U\times\Y\times\U} c(v,u) \, \dd\pi^{\star}(z,v,y,u), \\
&\int_{\Y\times \Y} \left[ \int_{\U \times \U} c(v,u) \, \dd\pi^{\star}(v,u \mid z, y) \right] 
\dd \delta( z- y) \dd \eta_\Y(z)
= \int_{\Y} \left[ \int_{\mathcal{U\times U}} c(v, u) \dd\pi^{y}(v,u) \right] \, \dd\RefM_{\Y}(y).
\end{align*}
where $\pi^{y}(v,u)$ is an optimal coupling of $\RefM(v\mid y)$ and $\TargetM(u\mid y)$ per the notation in 
the proof of \Cref{prop:conditional_kantorovich}. By \cref{ambrosio_hilbert_space} the coupling $\pi^y$ is induced by 
a unique OT map, i.e., $\pi^y = ( Id \times T_y ) \# \eta(\cdot \mid y)$. 
We now wish to argue that the maps $T_y$ can be concatenated together 
in such a way that the solution to \eqref{eq:constrained-Monge} is given by $T^\star(z,v) = (z,T^\star_\U(z,v))$ where 
$T^\star_\U(z, v) = T_z(v)$ $\RefM$-a.e. In order to see that this procedure results in a measurable map, note that $T_z(v) = 
\int_\U \pi^z(u, v) \dd u =\int_{\U} \wt{\pi}^\star(u,v\mid z)\dd u$. 
This yields the joint measurability of $T_z(v)$ in $z$ and $v$ by Fubini's theorem 
which implies that our candidate map $T^\star(z,v)$ is measurable. The optimality of $T^\star$ follows from the cost being equivalent to that of $\pi^\star$, the solution to the Kantorovich problem, which lower bounds the optimal Monge cost, yielding optimality. 
Furthermore, the maps $T^\star_\U(z, \cdot)$ are 
essentially unique due to \cref{rem:uniqueness}. This 
in turn implies the essential uniqueness of the map $T^\star$.
\end{proof}

\begin{remark}
    Note that our argument for the measurability of $T^\star$
    is brief and simple (for example, in comparison to \cite{carlier2016vector}). This is due to the fact 
    that we already dealt with measurable selection issues 
    in defining the solution to the Kantorovich problem. Since 
    the solution to that problem was obtained by concatenation of 
    couplings between the conditionals we were able to simply 
    integrate those conditionals for fixed $z$ to define the 
    maps $T_y$. 
\end{remark}

\begin{remark}
We highlight that our method of proof for \Cref{prop:conditional_monge_solvability}
mainly relies on \Cref{prop:conditional_kantorovich} 
and a solvability result for the Monge problems on 
the conditionals $\eta(\cdot \mid y)$ and $\nu(\cdot \mid y)$. This implies 
that the result can be generalized to broader choices of $\U$ and $c$
for which we can still establish the solvability of the Monge 
problems between the conditionals. A simple example would be 
the choice of costs of the form $c(z, v ; y, u) 
= \| z - y\|_\Y^p + \| u - v \|_\U^q$ for $p,q \in (1, + \infty)$.
\end{remark}

\subsection{Relaxing the conditional Kantorovich and Monge problems}\label{subsec:relaxedProblem}
In this section we consider a relaxation of the conditional Kantorovich \eqref{eq:Constrained Kantorovich} and Monge \eqref{eq:constrained-Monge} problems with quadratic costs to 
an unconstrained problem with a singularly perturbed cost function. Our result can be 
viewed as the Hilbert space
generalization of the main result of \cite{carlier2010knothe} where the authors 
showed that the KR map can be obtained as a limit of OT maps obtained from singularly perturbed 
costs. Our interest in this relaxation stems from algorithmic considerations. For example, 
constructing triangular OT maps by first transporting the conditionals of $\RefM, \TargetM$
is not practical in situations where we only have access to an empirical approximation 
to $\TargetM$ since in this case we may only have a single $u$ sample per each $y$ sample. 
This motivates the formulation of an OT problem between $\RefM, \TargetM$ directly, where the solution to the empirical OT problem converges to that of the true solution. 
Motivated by these applied considerations we present the following result:

\begin{proposition}\label{prop:conditional_limit}
Let $\U$ and $\Y$ be separable Hilbert spaces
and consider the cost $c(z, v,; y, u)  = c(v,u) = \| u - v \|_\U^2$. 
Suppose the conditions of  \Cref{prop:conditional_monge_solvability} hold for $p=2$
so that the conditional Kantorovich and Monge problems have unique solutions $\pi^\star$
and $T^\star$ respectively.
For $\epsilon>0$,  define
\begin{equation}\label{eq:perturbed-OT}
    K_{\epsilon}(\pi) := \int c_\epsilon \dd \pi,
\qquad c_\epsilon(z, v; y, u) := \| z - y\|_\Y^2 + \epsilon\| v - u\|_\U^2. 
\end{equation}
 Let $\pi_\epsilon:=\inf_\pi K_\epsilon (\pi)$. Then 
 (i)  the $\pi_\epsilon$ are unique and $\pi_\epsilon = (\Id \times T_\epsilon)\# \RefM$
 for a unique map $T_\epsilon$;
 (ii) $\pi_\epsilon \rightharpoonup \pi^\star$ (weak convergence) as $\epsilon\to0$; 
 and (iii) $T_\epsilon$ converges in $L^2_\RefM(\Y \times \U; \Y \times \U) $ to the triangular map $ T^\star$ defined in  \cref{prop:conditional_monge_solvability}  as $\epsilon \to 0$.
\end{proposition}
\begin{proof}
    As stated earlier our result is related to the main result in 
 \cite{carlier2010knothe} and follows a similar path at the beginning of 
 the proof before deviating towards the end where we rely on  \Cref{prop:Conditional Coupling Conditioning,prop:conditional_monge_solvability} instead of direct arguments. The key differences are that (1) here $\Y,\U$ are separable Hilbert spaces 
 while \cite{carlier2010knothe} considers finite-dimensional Euclidean spaces and (2)
 we work under \Cref{assump:y-marginals-match} which simplifies our proof.

Statement (i) follows directly from \Cref{ambrosio_hilbert_space} by observing that 
$\sqrt{c_\epsilon}$ is a norm defined on $\Y \times \U$ which is a separable Hilbert space. 

We now prove (ii), which consists the main technical aspects of the proof. 
Let 
$c_{\epsilon}(z,v,y,u)=\|z-y\|_\Y^{2}+\epsilon\|v-u\|_\U^{2}$
so that
$K_{\epsilon}(\pi)=\int c_{\epsilon}\,\dd\pi$
and set 
\begin{align*}
\pi_{\epsilon}:=\arg\min_{\pi}K_{\epsilon}(\pi), \qquad 
\pi_{0}:=\lim_{\epsilon\to0}\arg\min_{\pi}K_{\epsilon}(\pi),
\end{align*}
where the limit  exists (up to a subsequence) by the tightness of the set of couplings between fixed marginals. \cite[Lemma 4.4]{villani-OT}.
By the  optimality of $\pi_{\epsilon},$ we have that
\begin{equation}\label{disp-1}
    \int c_{\epsilon}\,\dd\pi_{\epsilon}\leq\int c_{\epsilon}\,\dd\pi^\star. 
\end{equation}
Towards 
taking the limit $\epsilon \to 0$ on both sides, 
note that the right hand side is continuous in $\epsilon$ by
the hypothesis that $\RefM$ and $\TargetM$ have finite second moments (so the 
integral is finite) and the 
fact that the measure $\pi^\star$ is fixed. 
On the other hand, we can verify that $\epsilon\mapsto\min_{\pi}K_{\epsilon}(\pi)$ is continuous since, 
\begin{equation}\label{disp-2}
\begin{aligned}
\left|\int c_{\epsilon}\,\dd\pi_{\epsilon}-\int c_{\epsilon'}\,\dd\pi_{\epsilon'}\right|
&=\left|\int c_{\epsilon}\,\dd\pi_{\epsilon}-\int c_{\epsilon'}\,\dd\pi_{\epsilon}+\int c_{\epsilon'}\,\dd\pi_{\epsilon}-\int c_{\epsilon'}\,\dd\pi_{\epsilon'}\right|, \\
&\leq\left|\int c_{\epsilon}\,\dd\pi_{\epsilon}-\int c_{\epsilon'}\,\dd\pi_{\epsilon}\right|+\left|\int c_{\epsilon'}\,\dd\pi_{\epsilon}-\int c_{\epsilon'}\,\dd\pi_{\epsilon'}\right|.
\end{aligned}
\end{equation}
Now let $\epsilon\to\epsilon'$. Once again by the 
compactness of the couplings we  know that 
$\pi_{\epsilon}\rightharpoonup\pi_{\epsilon'}$ up to a subsequence,
which means that the second term on the right hand side vanishes since the cost is fixed.
Additionally,  $|\int c_\epsilon \dd \pi_\epsilon - \int c_{\epsilon'} \dd \pi_\epsilon |
= |\epsilon - \epsilon'| \int_{\Y \times \U \times \Y \times \U}  
\| u - v\|^2_\U \pi( \dd z, \dd v; \dd y, \dd u) \le C |\epsilon - \epsilon'|  $
for a uniform constant $C > 0$ depending on the second moments of $\RefM, \TargetM$,
proving that the first term in the right hand side of \eqref{disp-2} also vanishes as 
$\epsilon \to \epsilon'$ allowing us to conclude that $\epsilon\mapsto\min_{\pi}K_{\epsilon}(\pi)$
is continuous. 
We can now take $\epsilon \to 0$ on both sides of \eqref{disp-1} 
to obtain 
\begin{equation}\label{disp-4}
\int c_{0}\,\dd\pi_{0}\leq\int c_{0}\,\dd\pi^{\star}.
\end{equation}
Now consider the operator $P_{\Y \times \Y}: (z, v; y, u) \to (z, y)$ 
and observe that 
since  $c_{0}(z,v,y,u)=\|z-y\|_\Y^{2}$ is independent of $v,u$ then
$P_{\Y \times \Y} \# \pi^\star$ is by definition 
the minimizer of $\int c_0 \dd \pi$ as it couples $z,y$ by the identity 
map and achieves zero cost. 
 Thus, by optimality of $\pi^\star$
we have that 
$\int c_{0}\,\dd\pi_{0}=\int c_{0}\,\dd\pi^{\star} =0.$
In turn this implies that 
$P_{\Y \times \Y}\# \pi_{0}=P_{\Y \times \Y}\# \pi^{\star} =: \pi^\star_{\Y \times \Y}$ 
since for any coupling $\pi$ we have that $\int c_0 \dd \pi = 
\int \| z - y \|_\Y^2  P_{\Y \times \Y} \# \pi(\dd z, \dd y)$
which is precisely the standard OT problem between the $\Y$ marginals of $\RefM$
and $\TargetM$ and has a unique solution.

Let us further define $P_{\U \times \U}: (z, v; y, u) \mapsto (v, u)$. 
 By the above observation and the fact that 
 $\pi_\epsilon$ is optimal for $c_\epsilon$ we can write,  
\begin{align*}
\epsilon &\int\|u-v\|_\U^{2}\,\dd\pi_{\epsilon}(z, v; u, u)
=\int\|z-y\|_\Y^{2}\, \dd \pi_{\Y \times \Y}^{\star}(z, y)
+\epsilon \int \|u-v\|_\U^{2}\, \dd P_{\U \times \U} \# \pi_{\epsilon}(v, u) \\
&\leq\int\|z-y\|_\Y^{2}+\epsilon\|u-v\|_\U^{2}\,\dd\pi_{\epsilon}(z,v; y, u) 
\leq\int \|z-y\|_\Y^{2}+\epsilon\|u-v\|_\U^{2}\,\dd\pi^{\star}(z, v; y, u) \\
&= \epsilon\int\|u-v\|_\U^{2}\,\dd\pi^{\star}(z, v; y, u).
\end{align*}
Hence, 
$\int\|u-v\|_\U^{2}\,\dd\pi_{\epsilon}\leq\int\|u-v\|_\U^{2}\,\dd\pi^{\star}$
for $\epsilon >0$ and so 
passing to the limit we get
\begin{equation}\label{disp-3}
    \int\|u-v\|_\U^{2}\,\dd\pi_{0}\leq\int\|u-v\|_\U^{2}\,\dd\pi^{\star}.
\end{equation} 
Notice that since $P_{\Y \times \Y} \# \pi_0 = \pi^\star_{\Y \times \Y} 
= (\Id \times \Id) \# \RefM_\Y$, then $\pi_0 \in \Pi_\Y$ and,
by \Cref{prop:Conditional Coupling Conditioning}, can be decomposed as 
$\pi_0( z, v; y,u) = \wt{\pi}_0(v, u \mid y) \delta(z- y) \RefM_\Y(y)$
where $\wt{\pi}_0 \in \Pi( \RefM(\cdot \mid y), \TargetM(\cdot \mid y))$.
This fact together with \eqref{disp-3} results in the inequality 
\begin{equation*}
    \int_\Y \left[ \int_{\U \times \U} \| u - v\|_\U^2 \dd \wt{\pi}_0(u, v \mid y)  \right]
    \dd \RefM_\Y(y) \le 
       \int_\Y \left[ \int_{\U \times \U} \| u - v\|_\U^2 \dd \wt{\pi}^\star(u, v \mid y)  \right]
    \dd \RefM_\Y(y).
\end{equation*}
Furthermore, the $\wt{\pi}^\star( \cdot, \cdot \mid y)$ 
are the unique optimal couplings of $\RefM(\cdot \mid y)$ and $\TargetM(\cdot \mid y)$ 
for a.e. $y$ which implies that 
$\wt{\pi}_0( \cdot, \cdot \mid y)$ should also be optimal for a.e. $y$ and so 
$\wt{\pi}_0 = \wt{\pi}^\star$ which completes the proof of the 
fact that $\pi_0 = \pi^\star$. Finally, (iii) follows immediately from $\cref{lem:mapconvergence}$.
{}
\end{proof}

\begin{remark}
    We note that while \Cref{prop:conditional_limit} is stated for the quadratic cost $c_\epsilon$, 
    the argument can be extended, with adjustment to the proof of \Cref{ambrosio_hilbert_space}, to 
    more general costs of the form $c_\epsilon(z, v,; y, u) = \| z - y\|_\Y^p + \epsilon \| v - u\|_\U^q$
    for exponents $p,q \in (1, +\infty)$. In fact, our proof of  (ii) remains unchanged.  
    One simply needs to verify the existence and  uniqueness of the $\pi_\epsilon$ in (i) and generalize \Cref{lem:mapconvergence} to  $L^p$ spaces for $p \in (1, + \infty)$ to generalize (iii). 
\end{remark}

\subsection{Extension to general triangular couplings and maps}\label{subsec:generic-triangular-OT}
Up to this point in this section we worked under
\cref{assump:y-marginals-match} that helped to simplify many of our arguments 
thanks to the fact that since the $\Y$ marginals of $\RefM$ and $\TargetM$ match, we can 
simply take the $T_\Y$ component of our triangular maps to be $\Id$. 
While this assumptions is sufficient for many practical applications, such as the BIPs considered in 
\Cref{sec:application}, there are still important problems where \Cref{assump:y-marginals-match}
does not hold and we may need to work with the broader family of triangular transport maps. 
In this section we present a generalization of our results to such a setting where 
$T_\Y$ itself is replaced with an OT map. Our result can be viewed as a generalization of KR rearrangements
to the case of Hilbert spaces. We present our result as a corollary since the proof follows 
directly from our previous results as well as standard results on the uniqueness of Monge maps.  

\begin{corollary}\label{cor:triangular}
Let $\Y,\U$ be separable Hilbert spaces. Consider the cost functions
Suppose that the reference marginal $\RefM_\Y$ and the reference conditionals $\RefM\left(\cdot \mid y \right)$ 
are regular and  $\RefM, \TargetM$ have finite  second moments 
Then it holds that
\begin{enumerate}[(i)]
    \item There exists a unique  map 
    $T^\star_\Y: \Y \to \Y$ that solves 
    \begin{equation*}
         \arg\min_{T_\Y} \int_{\Y} \| z - T_\Y(z) \|_\Y^2 \, \dd\RefM(z, v), \quad \st \quad T_\Y\#\RefM_\Y = \TargetM_\Y.
    \end{equation*}
    \item There exists a unique map $T^\star_\U: \Y \times \U \to  \U$ that solves
    \begin{equation*}
    \begin{aligned}
    \arg\min_{T_\U} & \int_{\Y} \| v - T_\U( z, v) \|_\U^2 \, \dd\RefM(z,v) \\ 
    & \st \quad 
    T(z, v) = (z, T_\U(z, v)), \quad T \circ (T_\Y^\star \times \Id) \# \RefM = \nu.         
    \end{aligned}
    \end{equation*}

    \item Define the map $T^\star: \Y \times \U \to \Y \times \U$ as $T^\star:= (\Id \times T^\star_\U) \circ (T^\star_\Y \times \Id)$. 
    If $T^\star_\Y$ is invertible \footnote{This will hold, for example, if $\TargetM_\Y$ is also 
    regular by \Cref{ambrosio_hilbert_space}(ii).}, then it holds that 
    $T^\star = \lim_{\epsilon \to 0} T_\epsilon$ in $L^2(\RefM)$ where $T_\epsilon$ are 
    the maps defined in \Cref{prop:conditional_limit}.
\end{enumerate}
\end{corollary}
\begin{proof}
Statement (i) is a direct consequence of \cref{ambrosio_hilbert_space} under our hypothesis on $\RefM_\Y$. 
To prove (ii) consider 
 $\overline{\RefM} = \left(T^\star_\Y \times \Id\right)\sharp\RefM$ and note that 
$\overline{\RefM}_\Y = \nu_\Y$. 
Then the result follows from  \Cref{prop:conditional_monge_solvability} 
as soon as we verify that $\overline{\RefM}$ satisfies the conditions of that theorem, in particular that $\overline{\RefM}(\cdot \mid y)$
are regular for all $y$.
To see this, let $A \subset \B(\U)$ 
be a $\RefM(\cdot \mid z)$ null set for all $z$, that is, $\RefM(A \mid z) = 0$. Then $\overline{\RefM}(A \mid y) =  
    \int_\Y \RefM( A \mid z) \delta( T^\star_\Y(z) - y ) \dd \RefM_\Y( z) = 0.$
To prove (iii) we need to slightly generalized the proof of \Cref{prop:conditional_limit}. 
In fact, the proof is identical up to \eqref{disp-4} by taking $\pi^\star$ to be 
the coupling induced by the optimal map $T^\star$ which achieves zero cost for $c_0$. 
Then we repeat the same arguments until the paragraph after \eqref{disp-3}
where we use \Cref{prop:Conditional Coupling Conditioning} with a nontrivial and invertible 
map $T^\star_\Y$. 
\end{proof}

\begin{remark}
    We note that this result can be further generalized to 
    the setting where the costs in (i) and (ii) are replaced with 
    $\| z - T_\Y(z) \|_\Y^p$ and $\| v - T_\U(z, v) \|_\U^q$ for 
    $p, q \in (1, + \infty)$ respectively. 
    Then we need $\RefM_\Y$  and
     $\RefM(\cdot \mid y)$ 
    to have finite $p$ and $q$ moments respectively.
\end{remark}

\section{Application to BIPs}\label{sec:application}

In this section we outline the application of our results from \Cref{sec:conditionalOT}
to the case BIPs, giving a likelihood-free characterization of  
posterior measures. Consider separable Hilbert spaces $\U, \Y$ and a 
forward map $F: \U \times \Y \to \Y$ with a model 
\begin{equation*}
    y = F(u, \xi), \qquad \xi \sim \gamma \in \PP(\Y),
\end{equation*}
that 
relates a  parameter $u$ to data/observations $y$ with 
possible observation noise  $\xi$.
The goal of Bayesian inference is to estimate the parameter $u$ given an instance 
of the data $y$. To achieve this, 
consider  a  prior measure $\mu \in \PP(\U)$, reflecting our prior belief about $u$, 
and define the joint measure 
\begin{equation}\label{Bayesian-joint-measure}
     \nu = \Law \big\{ (y, u)  \: : \: \text{draw } u \sim \mu, \: \xi \sim \gamma \text{ and set } 
     y = F(u, \xi) \big\} \in \PP( \Y \times \U). 
\end{equation}
Then the conditional $\nu(\cdot \mid y)$ is precisely 
the Bayesian posterior measure that updates our prior belief about $u$
given the observed data $y$. The standard BIP framework 
proceeds to identify the posterior via Bayes' rule \eqref{Bayes-rule-generic}
by its Radon-Nikodym 
derivative where
 the likelihood function $\exp \left( - \Phi(u; y) \right)$ 
 is proportional to  the density of $\nu( y \mid u) $ with respect to $\gamma$, 
 i.e.,  $\exp( - \Phi( u; y) ) \propto 
 \frac{\dd \nu( \cdot \mid u)}{\dd \gamma}(y)$; see 
 \cite[Thm.~6.29]{stuart-acta-numerica}
 for details. 
 
 While MCMC approaches \cite{stuart-mcmc, beskos-geometric-mcmc, hosseini2019two} for sampling the posterior measure rely on 
 the explicit form of the likelihood potential $\Phi$ in conjunction with the choice of the prior 
 $\mu$, here we aim to characterize the posterior $\nu( \cdot \mid y)$ using conditional 
 OT maps that circumvent the explicit form of the likelihood, thereby yielding a 
 likelihood-free inference method that does not even 
 require absolute continuity of the posterior with respect to the prior
 \cite{grelaud2009abc,gutmann2016bayesian,papamakarios2019sequential,MGAN}. 
To this end, we consider the reference measure $\eta = \nu_\Y \otimes \mu $, 
i.e., the independence coupling between the $\Y$-marginal of $\nu$ and the prior $\mu$. 
This choice of the reference $\eta$ is natural in the setting where 
 $F$ is known and samples from $\mu$ can be generated cheaply, since we have that $\nu_\Y = F\# (\mu \otimes \gamma) $ and so reference samples can be generated at the cost of 
 evaluating the forward map. We can then obtain the following corollary to 
 \Cref{prop:conditional_monge_solvability}. 

 \begin{corollary}\label{cor:prior-to-posterior-map}
     Suppose $\U, \Y$ are separable Hilbert spaces,  the prior  
     $\mu \in \PP(\U)$ is regular 
     according to \Cref{def:regular-measure}, $\gamma \in \PP(\Y)$, and 
      $F$ is continuous. 
      For $p \in (1, \infty)$ consider the conditional Monge problem 
      \begin{equation}\label{Bayesian-conditional-Monge}
          \inf_T \int_{\Y \times \U}  \| v - T_\U(z, v) \|_\U^p  \: \dd \eta( z, v)  
          \quad {\st }\quad T \#  \eta = \nu \text{ and } T(z,v) = (z, T_\U(z, v)).
      \end{equation}
      If the infimum  is finite and feasible then the solution $T^\star_\U$ 
      satisfies $T^\star_\U(y, \cdot) \# \mu = \nu(\cdot \mid y)$, $\nu_\Y$-a.e. 
 \end{corollary}

 \begin{proof}
     The continuity of $F$ ensures that $\nu$ is a well-defined Radon measure by 
     \cite[Thm.~9.1.1]{bogachev2}.  Then the corollary follows directly from 
     \Cref{prop:conditional_monge_solvability} with the 
     cost $c(z, v; y, u) = \| v - u\|_\U^p$ 
     and by observing that $\eta(\cdot \mid y)  = \mu$ which is assumed  regular.
 \end{proof}

We interpret \Cref{cor:prior-to-posterior-map} as a characterization of a family of 
prior-to-posterior maps $T^\star_\U(y, \cdot)$ parameterized by 
the data $y$. We  highlight a few important points: first, 
our choice of the reference $\eta$ is non-standard 
in comparison to other conditional generative models in the 
literature such as \cite{adler2018deep, MGAN} where $\eta = \nu_\Y \otimes \eta_\U$
with $\eta_\U$ taken to be standard Gaussian while here we take $\eta_\U = \mu$.  
Second, $\mu$ is completely arbitrary and so the transport-based characterization of 
posteriors naturally 
extends to non-Gaussian priors  
which is contemporary in the context of infinite-dimensional MCMC \cite{chen2018dimension, hosseini2019two, hosseini2023spectral}.

\subsection{Finite-dimensional approximations and perturbation analysis}\label{subsec:map-stability}
The use of conditional transport maps for practical infinite-dimensional 
inference requires the inevitable discretization of the problem 
to achieve a finite-dimensional approximation and so we turn our 
attention to such approximations of both the parameter $u$ and the
data $y$. Letting $\U^N$ denote an $N$-dimensional 
linear subspaces of $\U$ we consider  projection maps $P_\U^N: \U \to \U^N$, such as  
finite element discretizations of functions or projections onto 
Karhunen-Lo{\'e}ve modes. Writing $\mu^N : = P^N_\U \# \mu $  we can now define the 
approximate target measure
\begin{equation*}
\begin{aligned}
\nu^{N} & := \Law \left\{ (y, u)  \: : 
         \: \text{draw } u \sim \mu^N, \: \xi \sim  \gamma \text{ and set } 
      y =  F(u, \xi)  \right\} \in \PP( \Y \times \U),
\end{aligned}
\end{equation*}
with the associated posterior measures 
\begin{equation*}
    \frac{\dd \nu^N( \cdot \mid y) }{\dd \mu^N }(u) 
    = \frac{1}{Z^N(y)} \exp \left( - \Phi( u; y)  \right), 
    \qquad Z^N(y) = \int_\U  \exp \left( - \Phi(u; y) \right) \dd \mu^N(u).
\end{equation*}
The above measures are concentrated on the finite dimensional subspace $\U^N$
and hence amenable to numerical computations. Posing \eqref{Bayesian-conditional-Monge}
for the measure $\nu^N$ we obtain the discretized/approximate conditional Monge problem 
\begin{equation}\label{disp:discrete-conditional-Monge}
     \inf_T  \int_{\Y \times \U} \| v - T_\U(z, v) \|_\U^p \: \dd \eta(z, v)
     \quad 
     \st  
     \quad
     T\# \eta = \nu^N \text{ and } T(z, v) = ( z, T_\U(z, v)),
\end{equation}
with solution $T^N$. Since $T^N_\U(y, \cdot) \# \mu = \nu^N(\cdot \mid y)$ then 
we trivially obtain  the identity 
\begin{equation*}
    D( T^N_\U(y, \cdot) \# \mu, T^\star_\U(y, \cdot) \# \mu) = D( \nu^N, \nu),
\end{equation*}
for any statistical divergence or metric $D$ for a.e. $y$.
Controlling the right hand side of this expression in terms 
of the operator norm of $I - P^N_\U$
has 
been the subject of extensive study in the perturbation analysis of BIPs
\cite{sprungk2020local, garbuno2023bayesian}. 
We recall a simple version of such a result following \cite[Sec.~4.1, Cor.~3.6, and  Thm.~3.8]{garbuno2023bayesian}:

\begin{lemma}\label{lem:wasserstein-stability-posterior}
    Suppose $\Phi$ is positive, bounded, and bi-Lipschitz, and $\mu$
     has bounded moments of degree 2. Then there exist 
     constants $C_1, C_2 >0$ so that 
     \begin{equation*}
     \begin{aligned}
     W_1(\nu^N( \cdot \mid y), \nu( \cdot \mid y)) & \le C_1 W_2(\mu, \mu^N), && \forall N \ge 0, \\
     W_1(\nu(\cdot \mid y), \nu(\cdot \mid y') ) & \le C_2 
     \| y - y' \|_\Y, && \forall y,y' \in \Y. 
     \end{aligned}     
     \end{equation*}
\end{lemma}
Using the coupling $(\Id \times P^N_\U) \# \mu$
we can write $W_2^2(\mu, \mu^N) 
= \inf_\pi \int \| u - v \|^2_\U \dd \pi(u,v) 
\le \| \Id - P^N_\U \|_{L^2(\mu)}^2$ leading to the error bound 
\begin{equation*}
    W_1( T^N_\U(y, \cdot) \# \mu, T^\star_\U(y, \cdot) \# \mu) 
    \le C \| \Id - P^N_\U \|_{L^2(\mu)}.
\end{equation*}
This error bound is sufficient from the perspective of sampling since it quantifies the 
quality of samples coming from $T^N_\U$ as compared to the ground truth map $T^\star_\U$. However, it 
is also interesting to ask whether $T^N_\U$ is also a good approximation to $T^\star_\U$, such a 
result can be obtained in the finite-dimensional setting and for Brenier maps thanks to 
stability results for Wasserstein distances under perturbations \cite{berman2021convergence,merigot2020quantitative, delalande2021quantitative}. We state  such a result below. 

\begin{proposition}\label{prop:prior-to-posterior-Brenier-map-with-rates}
    Take $\U = \R^d$ and $c(z, v; y, u) = \| u - v\|_2^2$ and suppose that
    $\gamma \in \PP(\Y)$ and $F$ is continuous. Further suppose that: (a)
     the prior $\mu$ is supported on a compact and convex set $A \subset \R^d$
    and absolutely continuous with respect to the Lebesgue measure 
    with a density that is bounded from above and below on $A$; and (b) 
    the potential $\Phi$ is positive, bounded, and bi-Lipschitz. 
    Then for $\nu_\Y$-a.e. $y, y'$ and for $N \le d$, it holds that
    \begin{equation}\label{disp:map-rates}
        \begin{aligned}
        \| T^N_\U(y, \cdot) - T^\star_\U(y, \cdot) \|_{L^2(\mu)} & \le C_1 \| \Id - P^N_\U \|_{L^2(\mu)}^{1/6}, \\
        \| T^\star_\U(y, \cdot) - T^\star_\U(y', \cdot) \|_{L^2(\mu)} & \le C_2 \| y - y'\|_\Y^{1/6}.
        \end{aligned}
    \end{equation}
    Here the constants $C_1, C_2 \ge 0$ depend on $d, A, \Phi$ and the choice of  $\mu$. 
\end{proposition}

\begin{proof}
    Hypothesis (b) ensures that the posterior measures $\nu(\cdot \mid y)$ and 
    $\nu^N( \cdot \mid y)$ are well-defined. 
    The choice of  $c$ implies that the maps $T^\star_\U(y, \cdot)$ and 
    $T^N_\U(y, \cdot) $ are precisely the Brenier prior-to-posterior maps  for each $y$. 
    Since $\mu$ is assumed to be supported on $A$ then the posterior measures $\nu(\cdot \mid y)$
    are also supported on the same set and the same holds for $\mu^N$ and the corresponding $\nu^N(\cdot \mid y)$. This, together with hypothesis (a) allows us to use \cite[Thm.~4.2]{delalande2021quantitative}
    to get the bounds 
    \begin{equation}\label{wasserstein-stability-display}
        \begin{aligned}
            \| T^N_\U(y, \cdot) - T^\star_\U(y, \cdot) \|_{L^2(\mu)} 
            & \le C_1 W_1( T^N_\U(y, \cdot)\# \mu, T^\star_\U(y, \cdot)\# \mu)^{1/6} \\ 
            \| T^\star_\U(y, \cdot) - T^\star_\U(y', \cdot) \|_{L^2(\mu)} 
            & \le C_2 W_1( T^\star_\U(y, \cdot)\# \mu, T^\star_\U(y', \cdot)\# \mu)^{1/6}.
        \end{aligned}
    \end{equation}
    The desired result now follows from \Cref{lem:wasserstein-stability-posterior}.
\end{proof}

Before the end of this section we make a few remarks regarding the implications and possible generalizations of the above result.  
    We emphasize that the main hurdle in obtaining analogous error bounds for $T^\star_\U$ in the setting of 
infinite-dimensional $\U$ is that, to our knowledge, analogous stability results to \eqref{wasserstein-stability-display} have not been established in that setting. In fact, any divergence $D$ 
for which \Cref{lem:wasserstein-stability-posterior} can be established and 
satisfies a stability result of the form $\| T_1  - T_2 \|_{L^2(\mu)} \le 
C D(T_1\# \mu, T_2 \#\mu)^\alpha$ 
for OT maps $T_1, T_2$ and constants $C, \alpha \ge 0$ can be used 
establish a  result similar to \Cref{prop:prior-to-posterior-Brenier-map-with-rates}.

The second display in \eqref{disp:map-rates} implies the Lipschitz 
continuity of $T^\star_\U(z, v)$ along the $z$ coordinate which has important 
implications when it comes to the numerical approximations of this map
where one may choose to parameterize it using polynomials, kernels, or 
neural nets. Obtaining quantitative error bounds for such parameterizations 
will depend on establishing the regularity of $T^N_\U$ on the space $\Y \times \U$
inherited from $T^\star_\U$; see for example \cite{baptista2023approximation}. 

Finally, we note that our error analysis is done entirely in the {\it population} regime 
where we assume that the measures $\mu, \nu$ are known. In likelihoood-free inference 
applications, such as our experiments in \Cref{sec:numerics}, 
we only have access to an empirical approximation to $\nu$ arising from a 
finite collection of samples $\{ y_j, u_j \}_{j=1}^J \sim \nu$ in which case, 
the statistical consistency and sample complexity of the underlying maps should be established. 
This is an exciting and contemporary  direction for future research that we 
do not address here. 

\section{Numerical Experiments}\label{sec:numerics}
We dedicate this section to numerical experiments that demonstrate how our theoretical 
analysis throughout the paper can lead to algorithms for likelihood-free/simulation 
based inference. We emphasize that our focus is 
not on algorithmic development and performance but rather on proof of concept. 
In \Cref{subsec:algorithms} we outline two algorithms that we used for 
conditional simulation, one based on direct solution of OT problems and another 
using a modified GAN. \Cref{subsec:2D-toy-example} presents our experiments involving 
two dimensional benchmark data sets while \Cref{subsec:darcy} outlines 
a set of experiments for the Darcy flow inverse problem introduced as a motivating 
example in the introduction. Throughout the section we give 
a brief account of the setup of our experiments and algorithms but 
further details and hyper-parameter choices can be found 
in our online repository\footnote{\url{https://github.com/TADSGroup/ConditionalOT2023}}.

\subsection{Summary of algorithms}\label{subsec:algorithms}
We focus on the particular setting where we only have access to empirical 
samples $\{ (y_j, u_j) \}_{i=1}^J \sim \TargetM$ and write $\TargetM^J$ 
to denote corresponding empirical measures. Similarly, we write $\RefM^J$ 
to denote the empirical measure of the samples $\{ (y_j, v_j) \}_{j=1}^J \sim \RefM$
where the $y_j$ are copies of the $y$-coordinates of the target samples 
while the $v_j \sim \RefM_\U$ are drawn from an arbitrary reference. This particular 
construction of the reference and target measures ensures that $\RefM^J$ and 
$\TargetM^J$ exactly satisfy \Cref{assump:y-marginals-match}.
In both algorithms the output is a map $\widehat{T}_\U$ which approximates 
the OT conditioning map $T^\star_\U$ from \Cref{sec:conditionalOT}. 
Given this approximation we can generate an arbitrary number of conditional samples 
simply by drawing new reference samples $v \sim \RefM_\U$ (assumed to be accessible) 
and evaluating $\widehat{T}_\U(y, v)$ to get an approximate sample from $\TargetM(\cdot 
\mid y)$.

Our first algorithm follows from
\cref{prop:conditional_limit} and the observation that the maps $T_\epsilon$
approximate the triangular map $T^\star$. Given the empirical measures $\RefM^J, \TargetM^J$ we formulate the OT problem \eqref{eq:perturbed-OT} for a 
small value of $\epsilon$. Since we choose the reference and target measures 
to have the same number of samples, the solution to the discrete OT problem 
defines a map on the sample set. We then extend this map to 
the entire support of $\RefM$ by a simple interpolation strategy. The resulting 
algorithm is summarized in \Cref{alg:plugin}.
This algorithm is very simple to implement as it only requires solving a linear program and interpolating the result. In our experiments, we used a distance weighted two nearest neighbor interpolant and observed that this simple approach 
worked well in low-dimensional problems such as the toy examples 
in \Cref{subsec:2D-toy-example} but it suffers from the curse of dimensionality in 
high-dimensional problems and is unable to take advantage of the 
smoothness of transport maps.

\begin{algorithm}[htp]
\footnotesize
\caption{Plugin Conditional OT}
\begin{algorithmic}[1]\label{alg:plugin}
\State \textbf{Input:} Samples $\{ (y_j,u_j) \} _{j=1}^{N}\sim \TargetM$ and
relaxation parameter $\epsilon >0$ 

\State \textbf{Output:} Conditioning map $\widehat T_\U:
\Y\times\U \to \U$

\State Draw $J$ samples $\{v_j\}_{j=1}^{J} \sim \RefM_\U$ and 
form empirical measures $\RefM^J, \TargetM^J \in \PP(\U \times \V)$.

\State Solve the empirical OT problem defined by cost \eqref{eq:perturbed-OT}
between $\RefM^J$ and $\TargetM^J$ to obtain a discrete map.

\State Construct $\widehat{T}_\U$ to satisfy the interpolation 
constraints $\widehat{T}_\U( y_j, v_j) = u_j$ for $j=1, \dots, J$.

\end{algorithmic}
\end{algorithm}

Our second algorithm follows \Cref{prop:conditional_monge_solvability} 
and aims to approximate the map $T^\star_\U$ by a relaxation of the 
conditional Monge problem. The idea here is to parameterize $T_\U^\star$
using an expressive function approximator, a neural net in our case, and 
solve a problem of the form 
\begin{equation*}\label{MGAN-approach}
    \widehat{T}_\U= \arg \min _{T_\U \in \mathcal{T}} 
    \int_{\Y \times \U} \|v - T_\U(z, v)\|^2_\mathcal{U} \dd \RefM(z, v) + 
    \frac{1}{\lambda} D( T\# \RefM^J,\TargetM^J), 
     \: \st \: T(z, v) = (z, T_\U(z, v)),
\end{equation*}
where $\mathcal{T}$ is a parametric family of functions from $\Y \times \U \to \U$, $D$ is a divergence/discrepancy that distinguishes pairs of measures in $\PP(\Y \times \U)$, 
and $\lambda > 0$ is a regularization parameter. Simply put, the above approach amounts to solving 
the conditional Monge problem \Cref{conditional-Monge} for the empirical measures 
and by relaxing the 
constraint $T\# \RefM  = \TargetM$ to simply minimize a divergence between $T \# \RefM$
and $\TargetM$. We implement an instance of this approach with $D$ taken to be 
a Wasserstein GAN-GP loss \cite{arjovsky2017wasserstein,gulrajani2017improved}
and letting $\mathcal{T}$ to be a neural net class with PCA projection 
and reconstruction maps in the input and output layers in the setting of function 
space data following the setup of \cite{MGAN, PCA-Net}. Indeed, our algorithm 
is closely related to the Monotone GANs of \cite{MGAN} with the exception 
that we minimize the Monge loss while Monotone GANs enforce monotonicity of 
the transport map $\widehat{T}_\U$ using a penalty term. 
\begin{algorithm}[htp]
\footnotesize
\caption{Conditional Wasserstein Monge GAN (WaMGAN)}
\begin{algorithmic}[1]\label{alg:WaMGAN}
\State \textbf{Input:} Samples $ \{ (y_j,u_j) \} _{j=1}^{J} \sim \TargetM$, 
regularization parameter $\lambda >0$
\State \textbf{Output:} Conditioning map $\widehat T_\U:\Y\times\U\to\Y\times\U$
\State Initialize batch size $M$, critic iterations $n_{\text{critic}}$, learning rate $\alpha > 0$ 
\State Initialize neural net generator $T_\U^\theta : \Y\times \U \to \U$ parameterized 
by weights $\theta$
\State Initialize neural net critic $C^w : \Y\times \V \to \mathbb{R}$ parameterized by weights $w$
\For{number of training iterations}
    \State Sample a mini-batch $\{ (y_j,u_j) \} _{j=1}^{M} \sim \TargetM^J$
    \State Sample a mini-batch $\{ v_j \}_{j=1}^M$ from $\RefM_\U$ 
    and form $\{ (y_j,v_j) \} _{j=1}^{M}$
    \State Update the critic by ascending its stochastic gradient \footnotemark:
    \[
    w \gets w + \alpha \cdot  \nabla_{w} \left[ \frac{1}{M} \sum_{j=1}^{M} C^w(y_j,u_j) - \frac{1}{M} \sum_{j=1}^{M} C^w(y_j,T_{\theta}(y_j,v_j)) + R(C_w)\right]
    \]
    \State Update the generator by descending its stochastic gradient:
    \[
    \theta \gets \theta - \alpha \cdot \nabla_{\theta} 
    \left[  \frac{1 }{M} \sum_{j=1}^M \|v_j -T_\U^\theta(y_j,v_j)\|_\U^2 -\frac{1}{\lambda M} \sum_{j=1}^{M} C_w(y_j,T_\U^\theta(y_j,v_j)) \right]
    \]
\EndFor
\State  $\widehat T_\U \gets T_\U^\theta$
\end{algorithmic}
\end{algorithm}

\subsection{Two dimensional benchmarks}\label{subsec:2D-toy-example}
Here we consider a set of two dimensional benchmark examples from 
\cite{grathwohl2018ffjord} outlined in Figure~\ref{fig:2d-conditioning}. 
Our goal is to sample vertical slices of the target measure 
$\TargetM \in \PP([-1, 1]^2)$. We use $J=20000$ target samples 
and take $\RefM_\U = N(0,1)$. The results presented in 
\Cref{fig:2d-conditioning} were obtained using \Cref{alg:plugin}
with $\epsilon = 5 \times 10^{-3}$ and two nearest neighbor interpolation 
to extend the map to entire unit square. Another example of 
these results was presented in \Cref{fig:pinwheel-intro}.

We generally observe good agreement between the numerical and true 
conditional histograms across these examples with larger errors 
when the slices pass through a low density regions. We also observe that our maps $\widehat{T}_\U(y, \cdot)$
are noisy due to the two nearest neighbor interpolation. One 
can obtain a smoother map by using a different interpolant or using 
more neighboring points but this leads to further smoothing 
of the conditional histograms.
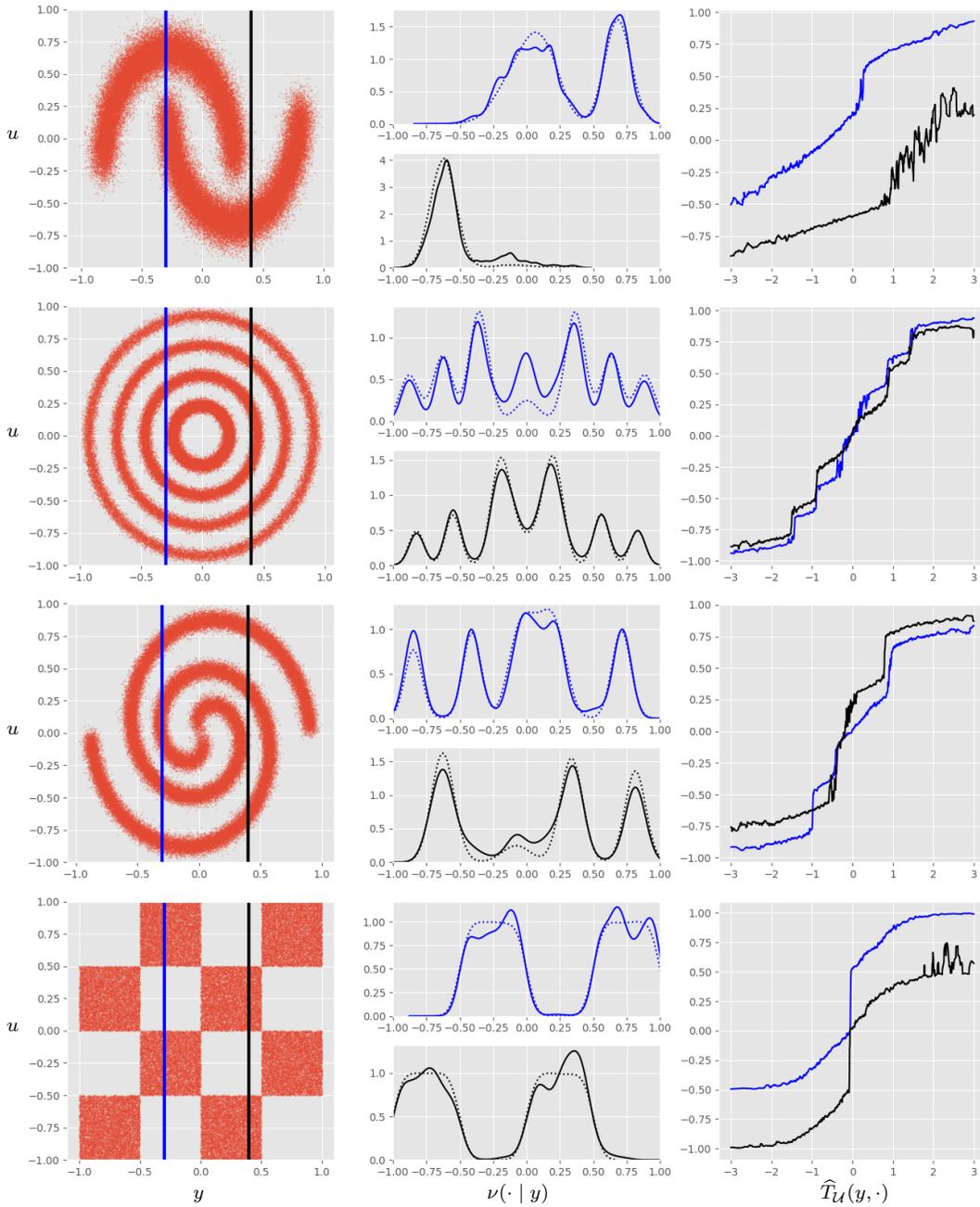
\begin{figure}[htp!]
    \centering
    \begin{overpic}[width=.9\textwidth]
    {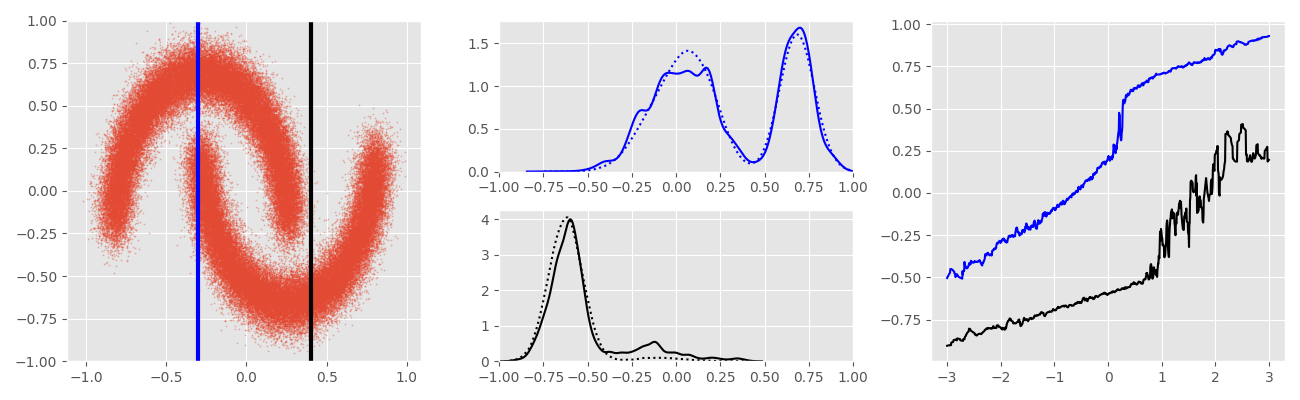}
        \put(-1, 16){\scriptsize $u$}
    \end{overpic}
    \begin{overpic}[width=.9\textwidth]
    {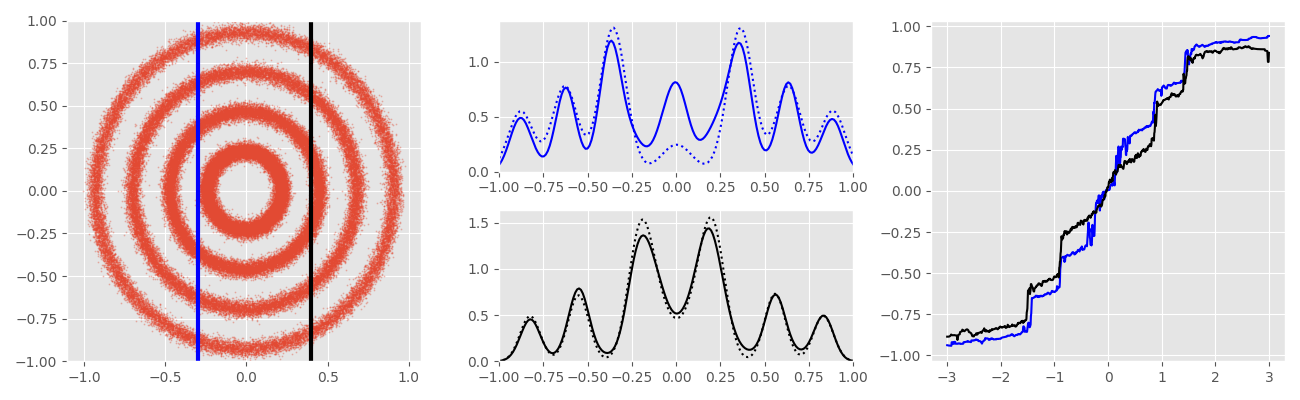}
        \put(-1, 16){\scriptsize $u$}
    \end{overpic}
    \begin{overpic}[width=.9\textwidth]
    {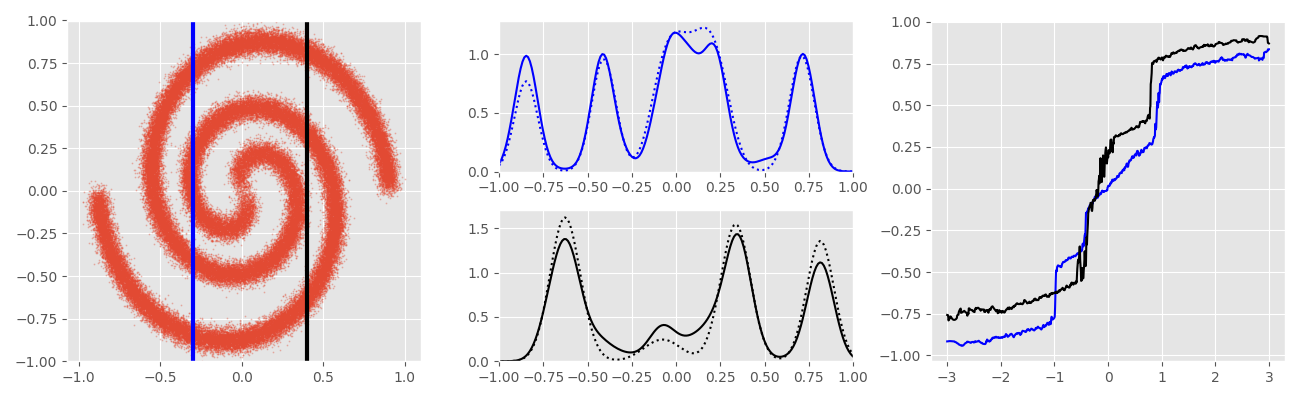}
        \put(-1, 16){\scriptsize $u$}
    \end{overpic}
    \begin{overpic}[width=.9\textwidth]
    {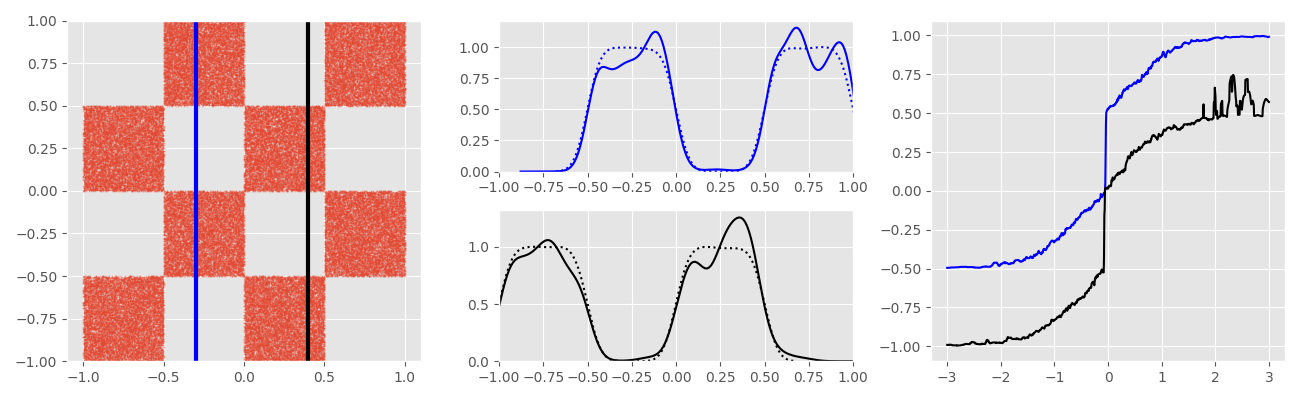}
        \put(-1, 16){\scriptsize $u$}
        \put(18, -1){\scriptsize $y$}
        \put(48, -1){\scriptsize $\TargetM(\cdot \mid y)$}
        \put(82, -1){\scriptsize $\widehat{T}_\U(y, \cdot)$} 
    \end{overpic}
    \caption{Conditional OT experiments on two dimensional benchmarks using 
    \Cref{alg:plugin}: (Left column)  the empirical samples from $\TargetM$. The 
    blue and black lines indicate slices along which we condition the measures. 
    (Middle column) comparison between the ground truth histogram
    indicated with dotted lines and 
    the conditional OT histogram indicated with solid lines. (Right column) 
    The conditioning maps $\widehat{T}_\U(y, \cdot)$ for each slice computed 
    using two nearest neighbor interpolation.
    }
    \label{fig:2d-conditioning}
\end{figure}


\subsection{Darcy flow Bayesian inverse problem}\label{subsec:darcy}
For our next example we consider the Darcy flow 
PDE in \eqref{Darcy-PDE}. Thanks to the Gaussian noise 
assumption for the data vector $y$ we obtain a quadratic likelihood
of the form 
\begin{equation}\label{darcy-likelihood}
    \Phi(u; y) = \frac{1}{2 \sigma^2} 
    \sum_{j=1}^m | p(u)(x_j) - y_j |^2, 
\end{equation}
where we recall $\sigma^2 >0$ is the measurement noise variance and 
we write $p(u)$ to denote the solution of \eqref{Darcy-PDE}
with a diffusion field $u$. 

We take the $x_j$'s (the measurement locations) to be a 
uniform $8\times 8$ grid over the unit square as shown in 
\Cref{fig:darcy-summary-stats} where we also show examples of $u^\dagger$
and $p(u^\dagger)$ for our experiments. 
We consider a 
Gaussian prior $\mu$ on $u$, in particular, we take $\mu = N(0, C)$
where $C$ is the Mat{\'e}rn covariance kernel
$C(x,y) =  \left(1 + \frac{\sqrt{3} |x-y|}{\ell}\right) \exp\left(-\frac{\sqrt{3} |x-y|}{\ell}\right)$ with  parameter 
$\ell = 1/2$. 

To solve the inverse problem 
we employed  \Cref{alg:WaMGAN} and compared its performance 
to the pCN algorithm \cite{stuart-mcmc} with which we take to be the 
ground truth. 
To train our map we used $J= 10^5$ samples that were 
generated by drawing $u_j \sim \mu$, computing $p(u_j)$
by solving the PDE using the FEniCSx package \cite{AlnaesEtal2015} 
on a fine mesh 
and  using interpolation to compute
$y_j = (p(u_j)(x_1), \dots, p(u_j)(x_{64}) ) + \epsilon_j$ 
with $\epsilon_j \sim N(0, \sigma^2 I)$; 
see our Github repository for further details of parameter 
choices for the solver and \Cref{alg:WaMGAN}.
In order to obtain mesh invariance, we first performed PCA on 
the  $u_j$ and applied the GAN to the latent PCA modes
to obtain conditional samples of the PCA modes of $u$ given $y$. 
We tuned the step size of pCN 
to obtain an average acceptance rate of $0.25$ 
after burn-in and used $5\cdot10^6$ iterations to ensure 
the chains had converged and the mean and variances were accurate and stable. We did not use PCA in our implementation of pCN since 
we could directly work with the covariance kernel of the prior.

In \Cref{fig:darcy-summary-stats} we present a summary of four instances of the 
Darcy problem. Here we drew four $u^\dagger$'s from the prior and generated synthetic data $y$. We solved the inverse 
problem using pCN for each data, taken to be the ground truth. The WaMGAN map
$\widehat{T}_\U$ was trained only once and was simply evaluated for each 
value of $y$ to generate posterior samples and compute statistics. This 
highlights the efficiency of triangular maps since each pCN run required a large 
number of forward map evaluations while WaMGAN only used $10^5$ evaluations. We also 
took care to present instances of the problem that are representative of the 
transport approach. In all four examples we see good agreement between 
the mean of WaMGAN and pCN, the former often captures the shape and features
of the mean but misses the amplitude sometimes. The variances are more smoothed out 
by WaMGAN but overall we see good agreement between the two methods in that 
regard as well. Further insight into the variance can be gained 
from \Cref{fig:variance-scatter} where we plot the pointwise variance
of WaMGAN posterior against pCN. Ideally we would like the data to lie on the 
diagonal line but we see a spread of variances and overall we see good 
agreement between the two methods but some spread exists indicating that WaMGAN 
can both over and under estimate variances in each instance of the problem. 
Finally, in \Cref{fig:darcy-samples} we compare five independent posterior samples 
from the two algorithms for one instance of the inverse problem. We see that 
these samples are qualitatively very similar, indicating that WaMGAN respects the 
regularity of the samples. We also observed that the Monge penalty  
allowed WaMGAN to inherit a near monotonicity property, in the sense that 
$\langle \widehat{T}_\U(y,z_1) - \widehat{T}_\U(y,z_2), z_1-z_2 \rangle \geq 0$ for over $99\%$ of the samples in the data, with $\lambda = 0.1$. However, 
there appears to be a tradeoff between monotonicity and quality of 
posterior statistics. We found that with $\lambda = 0.05$ we could obtain better 
posterior samples, mean, and variance but this resulted in a map that was 
monotone over a smaller fraction  of the training data.

\begingroup

\setlength{\tabcolsep}{4pt} 
\renewcommand{\arraystretch}{.8} 

\begin{figure}[htp]
\centering
\begin{tabular}{cccccc}
  \begin{overpic}[width=0.14\textwidth]{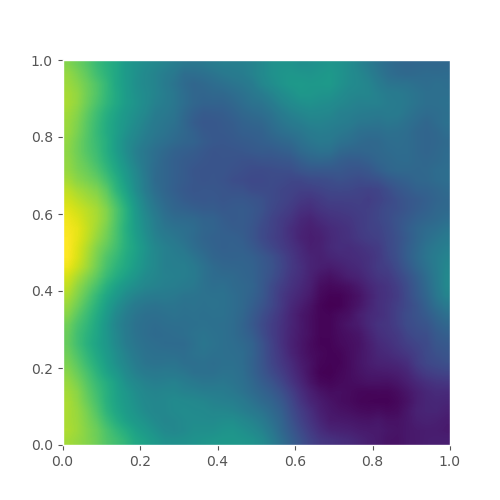}
  \end{overpic} &
  \begin{overpic}[width=0.14\textwidth]{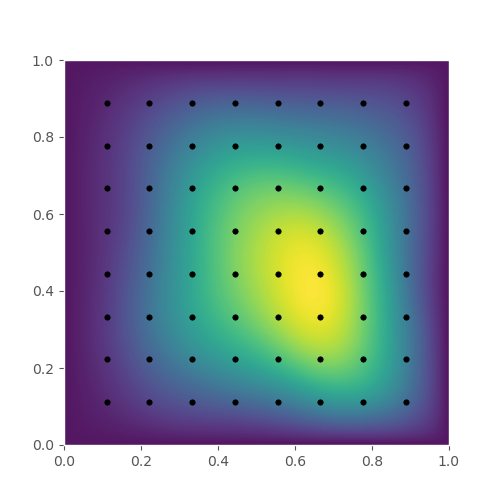} 
  \end{overpic} &
  \begin{overpic}[width=0.14\textwidth]{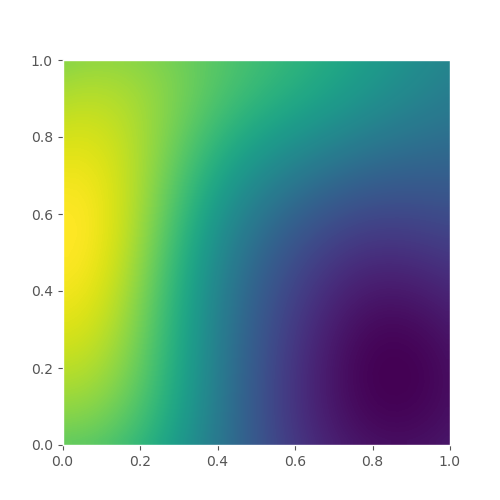}
  \end{overpic} &
\begin{overpic}[width=0.14\textwidth]{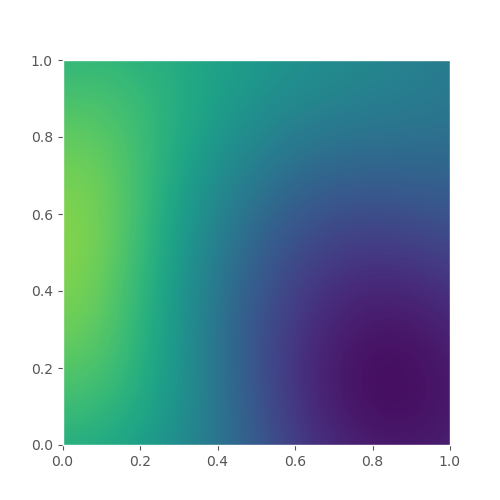}
  \end{overpic} &
  \begin{overpic}[width=0.14\textwidth]{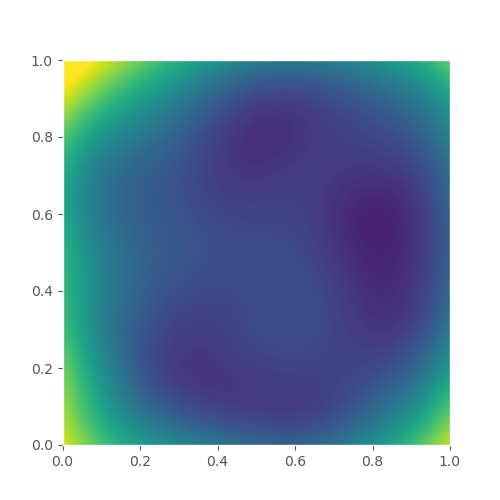}
    \end{overpic} &
  \begin{overpic}[width=0.14\textwidth]{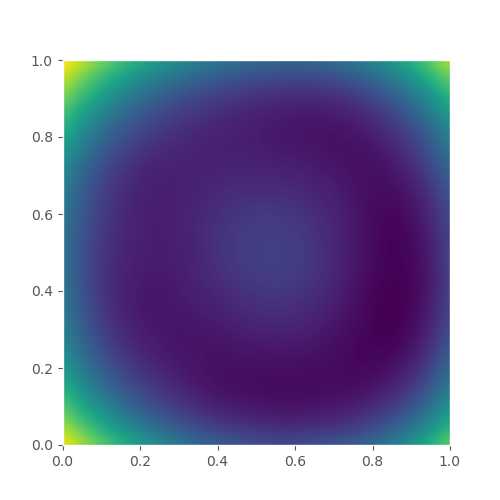}
  \end{overpic} \\ 
  \begin{overpic}[width=0.14\textwidth]{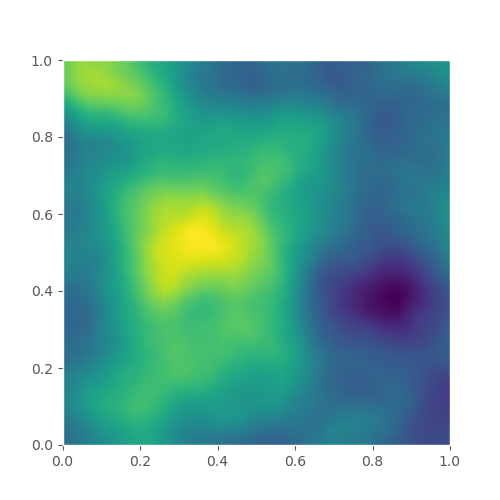}
  \end{overpic} &
  \begin{overpic}[width=0.14\textwidth]{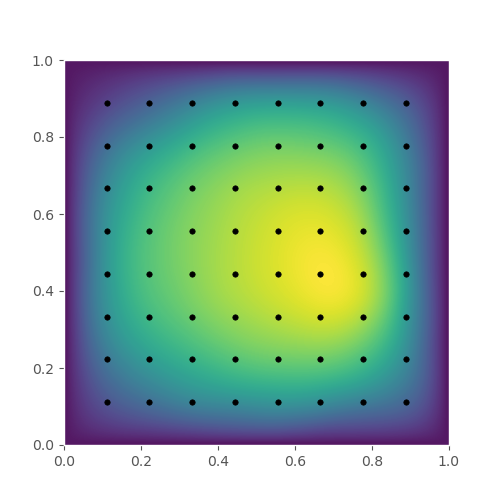} 
  \end{overpic} &
  \begin{overpic}[width=0.14\textwidth]{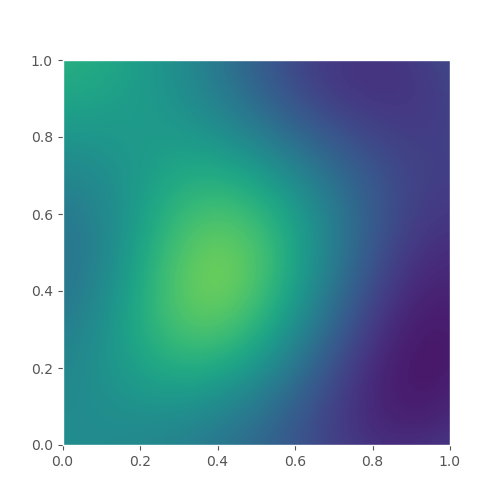}
  \end{overpic} &
    \begin{overpic}[width=0.14\textwidth]{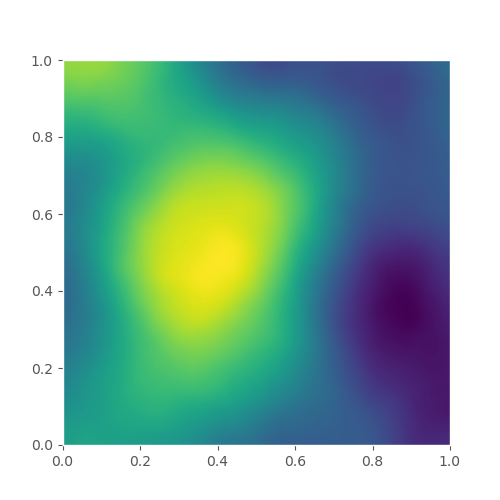}
  \end{overpic} &
  \begin{overpic}[width=0.14\textwidth]{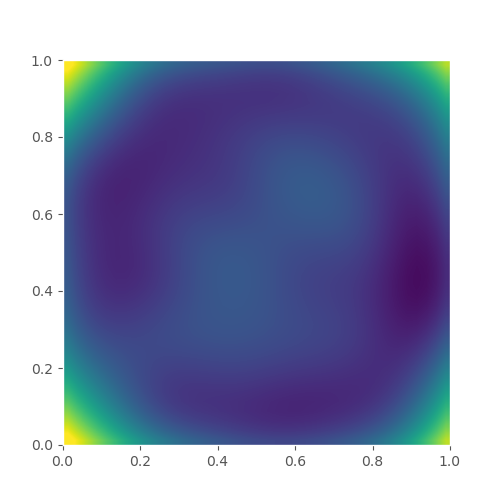}
    \end{overpic} &
  \begin{overpic}[width=0.14\textwidth]{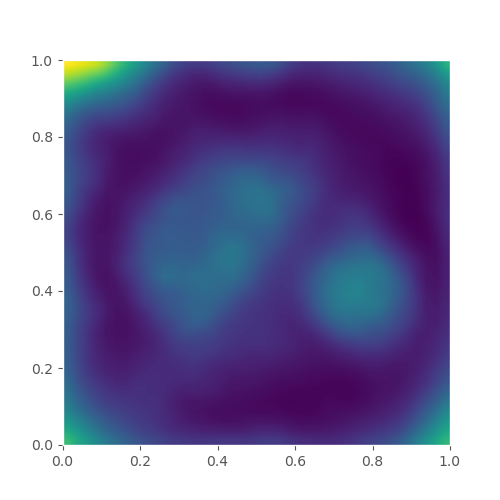}
  \end{overpic} \\ 
  \begin{overpic}[width=0.14\textwidth]{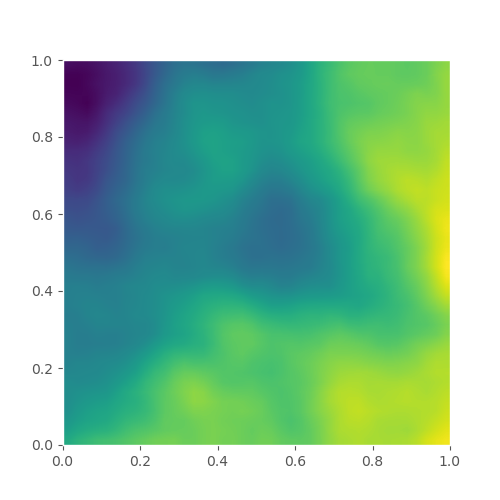}
  \end{overpic} &
  \begin{overpic}[width=0.14\textwidth]{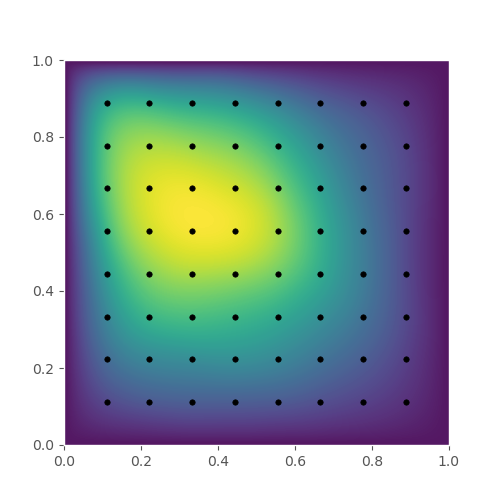} 
  \end{overpic} &
  \begin{overpic}[width=0.14\textwidth]{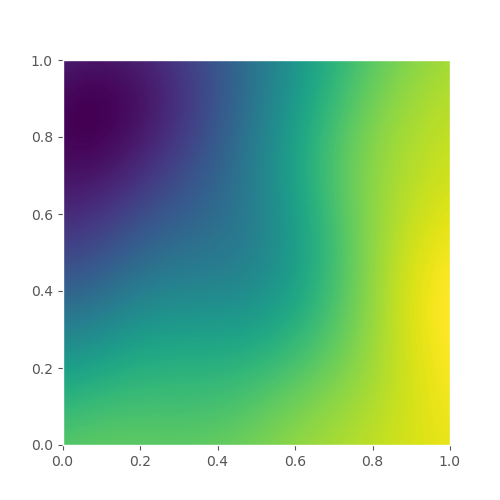}
  \end{overpic} &
  \begin{overpic}[width=0.14\textwidth]{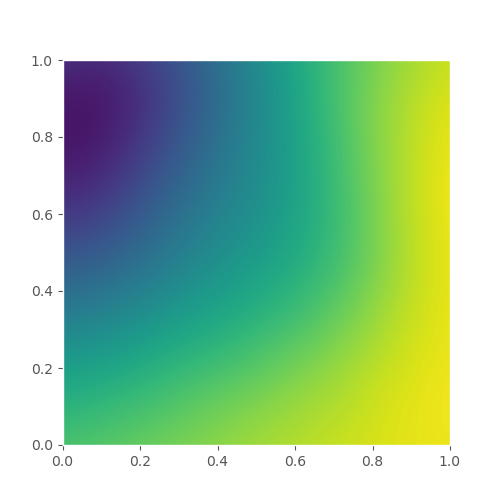}
  \end{overpic} &
  \begin{overpic}[width=0.14\textwidth]{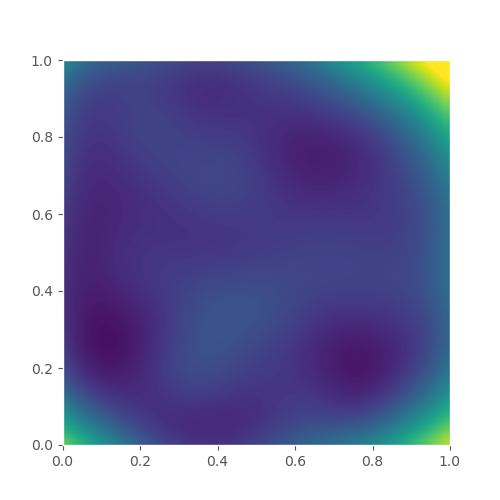}
    \end{overpic} &
  \begin{overpic}[width=0.14\textwidth]{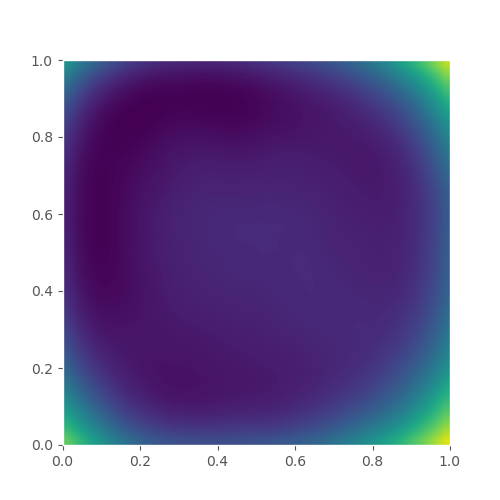}
  \end{overpic} \\ 
  \begin{overpic}[width=0.14\textwidth]{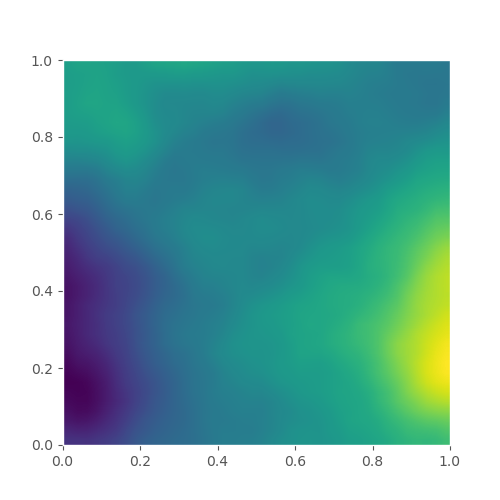}
  \put(48, -10){\footnotesize $u^\dagger$}
  \end{overpic} &
  \begin{overpic}[width=0.14\textwidth]{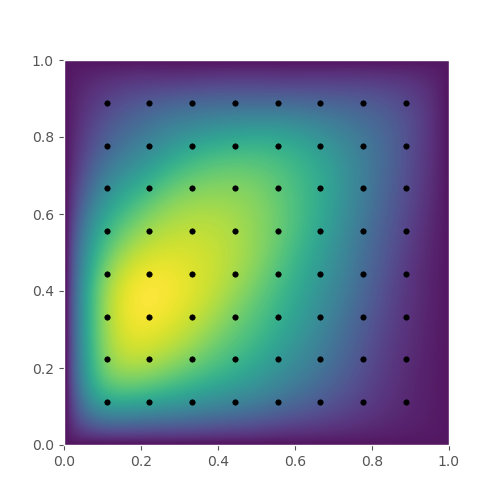} 
     \put(38, -10){\footnotesize $p(u^\dagger)$}
  \end{overpic} &
  \begin{overpic}[width=0.14\textwidth]{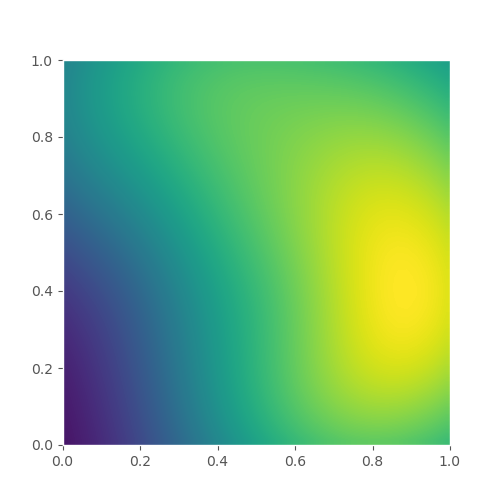}
     \put(30, -15){\scriptsize $\stackrel{\text{WaMGAN}}{\text{mean}}$}
  \end{overpic} &
  \begin{overpic}[width=0.14\textwidth]{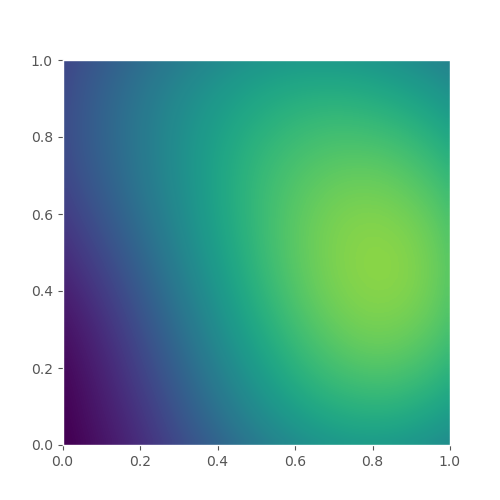}
  \put(35, -15){\scriptsize $\stackrel{\text{pCN}}{\text{mean}}$}
  \end{overpic} &
  \begin{overpic}[width=0.14\textwidth]{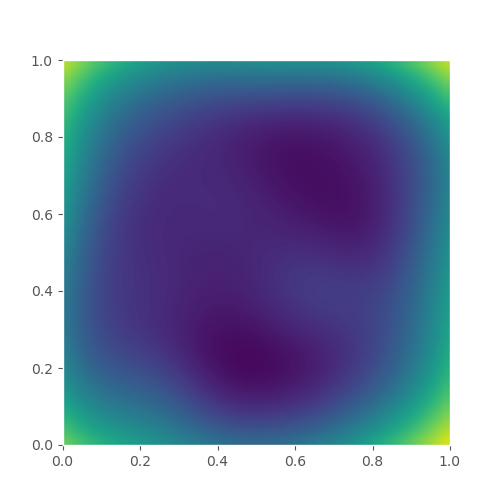}
  \put(30, -15){\scriptsize $\stackrel{\text{WaMGAN}}{\text{variance}}$}
    \end{overpic} &
  \begin{overpic}[width=0.14\textwidth]{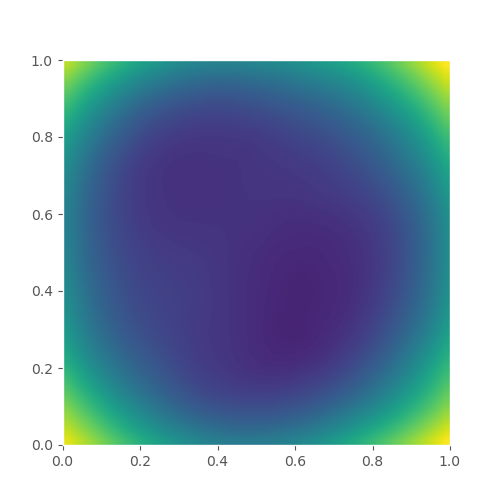}
  \put(30, -15){\scriptsize $\stackrel{\text{pCN}}{\text{variance}}$}
  \end{overpic}
 
\end{tabular}
\vspace{1ex}
\caption{Summary of our experiments for Darcy flow inverse 
problem comparing pCN and WaMGAN. From left, the first 
column shows the ground truth field $u^\dagger$, second column 
shows the corresponding $p(u^\dagger)$ and measurement 
locations $x_j$ denoted as black dots. Third and fourth columns 
show posterior means for WaMGAN and pCN respectively while 
the fifth and sixth columns show pointwise variances of the posteriors in the same order. In all cases we used the same color range for pCN 
and WaMGAN to ensure proper comparison between the two methods.}
\label{fig:darcy-summary-stats}
\end{figure}

\endgroup

\begin{figure}
    \centering
    \begin{overpic}[width=0.23\textwidth]{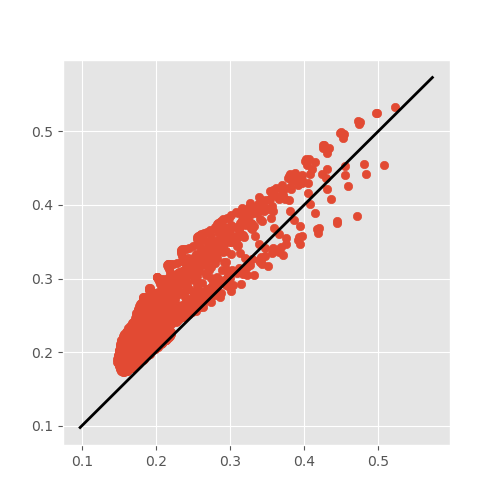}
    \put(20,-1.5){\footnotesize MCMC Variance}
    \put(-4,10){\footnotesize \rotatebox{90}{WaMGAN Variance}}
    \end{overpic}
    \begin{overpic}[width=0.23\textwidth]{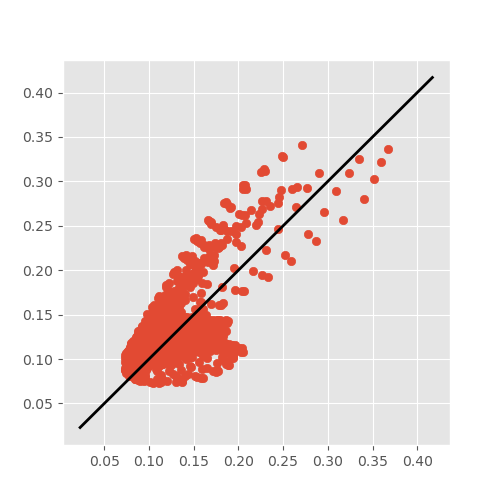}
    \put(20,-1.5){\footnotesize MCMC Variance}
    \end{overpic}
    \begin{overpic}[width=0.23\textwidth]{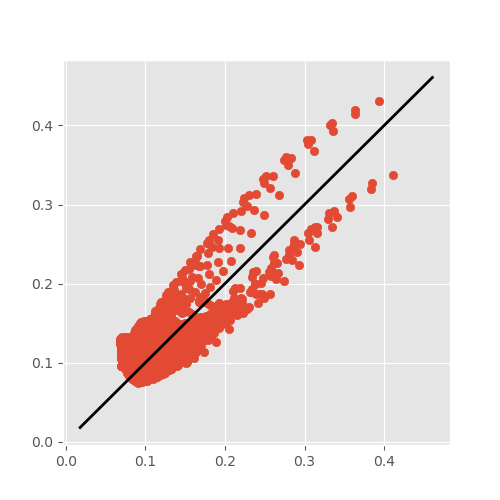}
    \put(20,-1.5){\footnotesize MCMC Variance}
    \end{overpic}
    \begin{overpic}[width=0.23\textwidth]{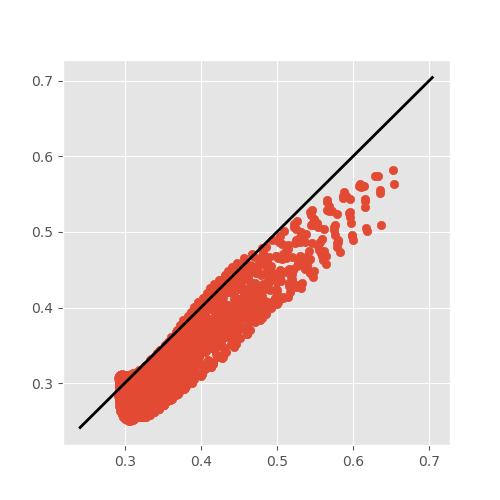}
    \put(20,-1.5){\footnotesize MCMC Variance}
    \end{overpic}
    \caption{Scatterplot comparison of pointwise posterior variances calculated 
    by WaMGAN and pCN for the four 
    examples in \Cref{fig:darcy-summary-stats}. The diagonal line represents perfect agreement between the two
    methods. }
    \label{fig:variance-scatter}
\end{figure}

\begingroup

\setlength{\tabcolsep}{4pt} 
\renewcommand{\arraystretch}{.8} 

\begin{figure}
    \centering
    \begin{tabular}{cccccc}
    \begin{overpic}[width=0.14\textwidth]{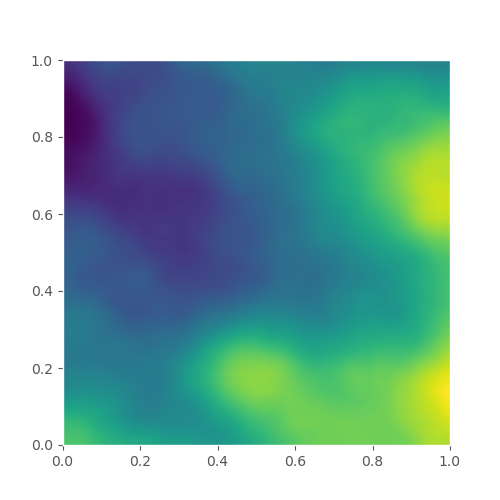}
    \put(-9,30){\footnotesize \rotatebox{90}{pCN}}
    \end{overpic} & 
    \begin{overpic}[width=0.14\textwidth]{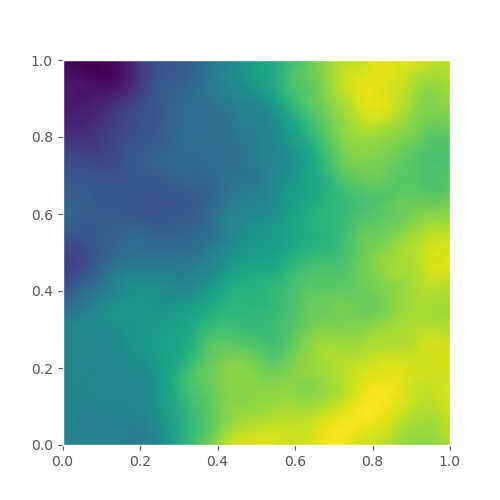}
    \end{overpic} & 
    \begin{overpic}[width=0.14\textwidth]{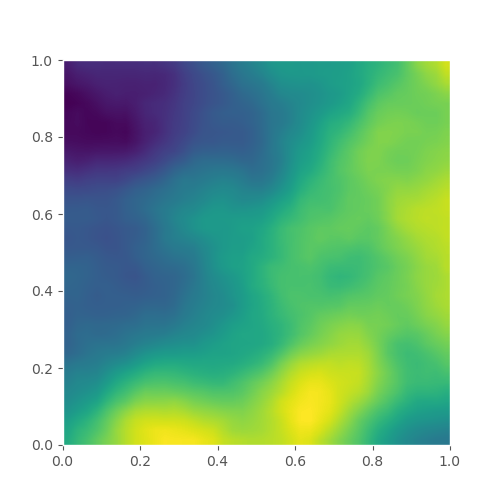}
    \end{overpic} & 
    \begin{overpic}[width=0.14\textwidth]{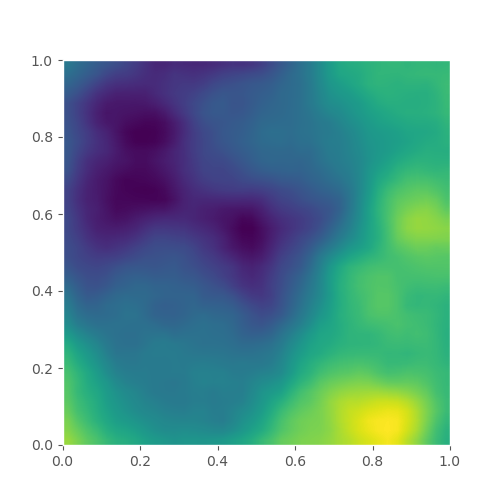}
    \end{overpic} & 
    \begin{overpic}[width=0.14\textwidth]{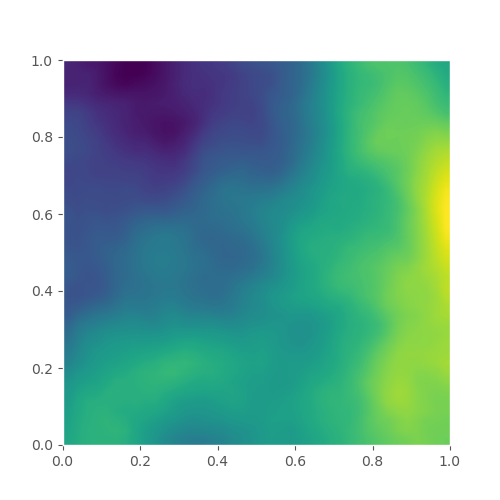}
    \end{overpic} & 
    \begin{overpic}[width=0.14\textwidth]{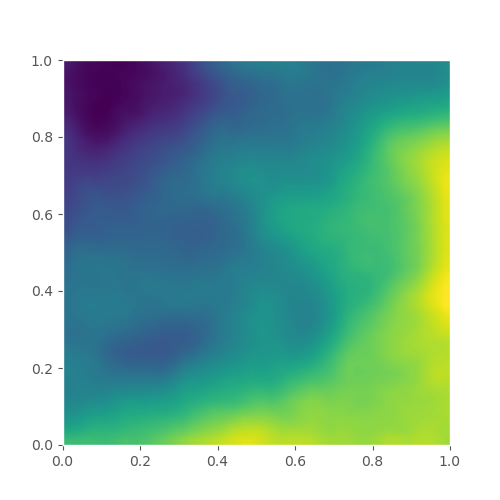}
    \end{overpic} \\ 
    \begin{overpic}[width=0.14\textwidth]{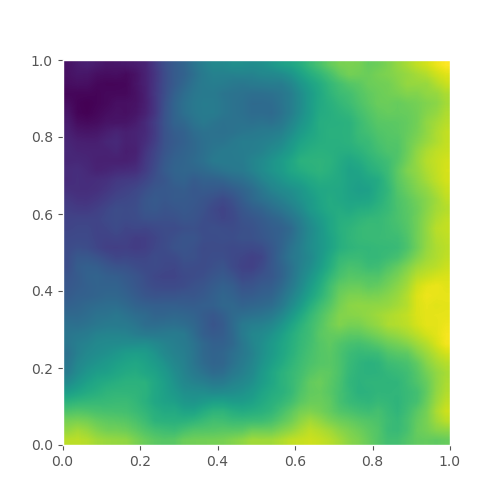}
       \put(-9,20){\footnotesize \rotatebox{90}{WaMGAN}}
    \end{overpic} &
    \begin{overpic}[width=0.14\textwidth]{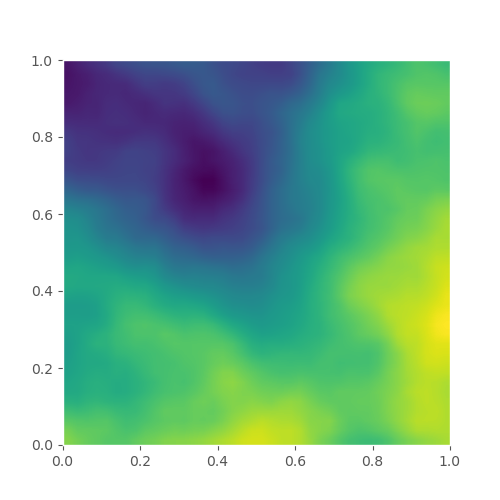}
    \end{overpic} &
    \begin{overpic}[width=0.14\textwidth]{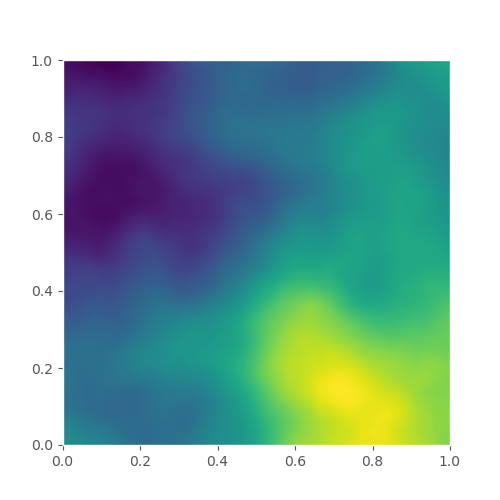}
    \end{overpic} &
    \begin{overpic}[width=0.14\textwidth]{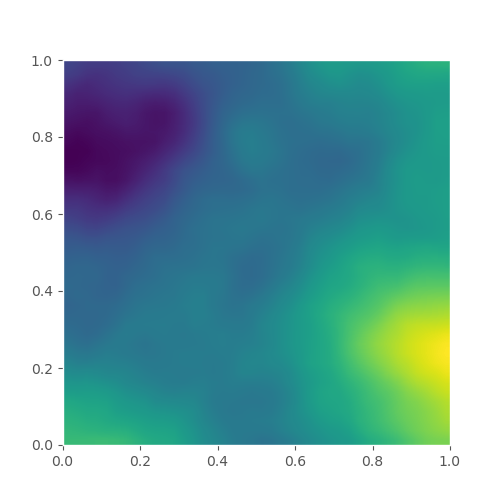}
    \end{overpic} &
    \begin{overpic}[width=0.14\textwidth]{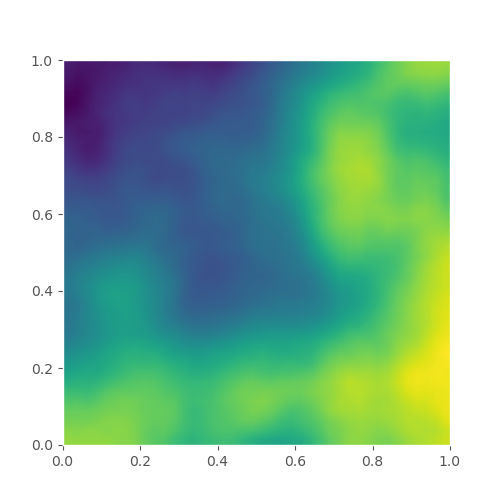}
    \end{overpic} &
    \begin{overpic}[width=0.14\textwidth]{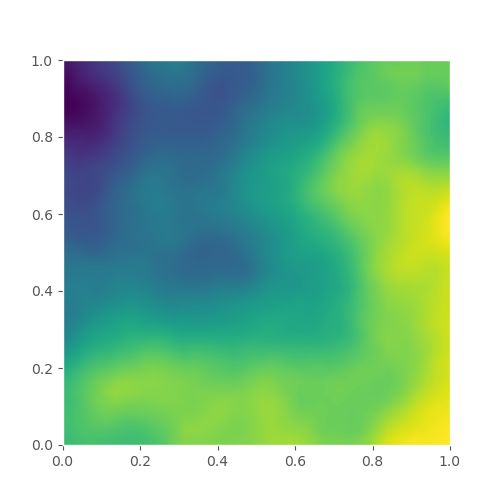}
    \end{overpic} 
    \end{tabular}

    \caption{A comparison of posterior samples from pCN (top row) vs WaMGAN (bottom row)
    for one of our experiments. }
    \label{fig:darcy-samples}
\end{figure}

\endgroup




\section{Conclusions}\label{sec:conclusion}

We presented a systematic analysis of conditional OT problem underlying triangular 
transport problems towards conditional simulation and likelihood-free/amortized 
inference problems. We focused our analysis on infinite-dimensional spaces, and 
in particular separable Hilbert spaces, that are useful for PDE inverse problems. 
We introduced the conditional Kantorovich and Monge problems and studied their 
solvability and characterized their solutions as a limit of 
unconstrained OT problems with a singularly perturbed cost and further 
tailored our theory to the case BIPs. Finally, we presented two numerical 
experiments that showed how our analysis could 
inform the design of algorithms for amortized inference. 

While we tackled the foundational theory of conditional OT in this work, there 
are many open questions and directions that remain. For example, 
understanding the regularity of conditional OT maps is an interesting 
and important question in practice, and in particular 
the regularity of the $T_\U$ map with respect to the $y$ variable. 
We gave some preliminary results in this direction in \Cref{subsec:map-stability}, 
showing that the maps are Lipschitz in $y$, but our results 
are limited to the finite-dimensional setting and for approximation theory, we 
may need more regularity. Another related question of interest is the 
sample complexity of triangular maps and in particular the $T_\U$
component. The question of interest is {\it how close is $T_\U(y, \cdot) \# \RefM_\U$
to the target conditional $\TargetM(\cdot \mid y)$?} This is a nontrivial 
question because having $T\# \RefM$ close to $\TargetM$ (on the product space) 
does not imply closeness of the conditionals. Furthermore, our experiments in 
\Cref{subsec:2D-toy-example} indicate that conditioning in low-density areas of 
$\TargetM$ can lead to higher errors. Finally, a major practical question is 
that of the design of algorithms for amortized inference based on triangular 
maps that are efficient and robust. Our experiments indicate that 
solving OT problems as linear programs yields poor convergence rates in high dimensions. However, working directly with the Monge formulation 
allowed us to leverage the approximation power of neural nets and obtain 
good results in a Hilbert space setting. While we used a simplistic 
relaxation of the Monge problem for this task, the question remains open as 
to what is the "best" relaxation of the Monge problem that leads to good 
empirical performance and approximates the conditional OT map;
see \cite{MGAN, al2023optimal,  alfonso2023generative, hosseini-taghvaei-OTPF-2, wang2023efficient} for other works in this direction.

\section*{Acknowledgments}
BH and AH are supported by NSF grant DMS-208535, {\it "Machine Learning for Bayesian Inverse Problems"}. AH was also partially supported by the NSF Infrastructure grant DMS-2133244, {\it "Pacific Interdisciplinary Hub on Optimal Transport"} made to the PIMS Kantorovich Initiative.
AT is supported by NSF grant EPCN-2318977, {\it "Variational Optimal Transport Methods for Nonlinear Filtering"}.

\appendix

\bibliographystyle{siamplain}
\bibliography{ConditionalOT_References}

\begin{thebibliography}{100}

\bibitem{adler2018deep}
{\sc J.~Adler and O.~{\"O}ktem}, {\em Deep {B}ayesian inversion}, arXiv
  preprint:1811.05910,  (2018).

\bibitem{al2023optimal}
{\sc M.~Al-Jarrah, B.~Hosseini, and A.~Taghvaei}, {\em Optimal transport
  particle filters}, arXiv preprint arXiv:2304.00392,  (2023).

\bibitem{hosseini-taghvaei-OTPF-2}
{\sc M.~Al-Jarrah, N.~Jin, B.~Hosseini, and A.~Taghvaei}, {\em Optimal
  transport-based nonlinear filtering in high-dimensional settings}, arXiv
  preprint arXiv:2310.13886,  (2023).

\bibitem{albergo2023stochastic}
{\sc M.~S. Albergo, N.~M. Boffi, and E.~Vanden-Eijnden}, {\em Stochastic
  interpolants: A unifying framework for flows and diffusions},
  arXiv:2303.08797,  (2023).

\bibitem{alfonso2023generative}
{\sc J.~Alfonso, R.~Baptista, A.~Bhakta, N.~Gal, A.~Hou, I.~Lyubimova,
  D.~Pocklington, J.~Sajonz, G.~Trigila, and R.~Tsai}, {\em A generative flow
  for conditional sampling via optimal transport}, arXiv preprint
  arXiv:2307.04102,  (2023).

\bibitem{AlnaesEtal2015}
{\sc M.~S. Alnaes, J.~Blechta, J.~Hake, A.~Johansson, B.~Kehlet, A.~Logg,
  C.~Richardson, J.~Ring, M.~E. Rognes, and G.~N. Wells}, {\em The {FEniCS}
  project version 1.5}, Archive of Numerical Software, 3 (2015).

\bibitem{ambrosio2005gradient}
{\sc L.~Ambrosio, N.~Gigli, and G.~Savar{\'e}}, {\em Gradient flows: in metric
  spaces and in the space of probability measures}, Springer Science \&
  Business Media, 2005.

\bibitem{amos2022amortizing}
{\sc B.~Amos}, {\em On amortizing convex conjugates for optimal transport}, in
  The Eleventh International Conference on Learning Representations, 2022.

\bibitem{arjovsky2017wasserstein}
{\sc M.~Arjovsky, S.~Chintala, and L.~Bottou}, {\em Wasserstein generative
  adversarial networks}, in International Conference on Machine Learning, PMLR,
  2017, pp.~214--223.

\bibitem{baldassari2023conditional}
{\sc L.~Baldassari, A.~Siahkoohi, J.~Garnier, K.~Solna, and M.~V. de~Hoop},
  {\em Conditional score-based diffusion models for {Bayesian} inference in
  infinite dimensions}, arXiv preprint arXiv:2305.19147,  (2023).

\bibitem{baptista2023approximation}
{\sc R.~Baptista, B.~Hosseini, N.~B. Kovachki, Y.~M. Marzouk, and A.~Sagiv},
  {\em An approximation theory framework for measure-transport sampling
  algorithms}, arXiv:2302.13965,  (2023).

\bibitem{baptista-function-space-score}
{\sc R.~Baptista and N.~B. Kovachki}, {\em Score-based diffusion models in
  function space}, arXiv:2302.07400,  (2023).

\bibitem{baptista2020adaptive}
{\sc R.~Baptista, O.~Zahm, and Y.~Marzouk}, {\em An adaptive transport
  framework for joint and conditional density estimation}, arXiv
  preprint:2009.10303,  (2020).

\bibitem{batzolis2021conditional}
{\sc G.~Batzolis, J.~Stanczuk, C.-B. Sch{\"o}nlieb, and C.~Etmann}, {\em
  Conditional image generation with score-based diffusion models}, arXiv
  preprint arXiv:2111.13606,  (2021).

\bibitem{berman2021convergence}
{\sc R.~J. Berman}, {\em Convergence rates for discretized {Monge--Amp{\`e}re}
  equations and quantitative stability of optimal transport}, Foundations of
  Computational Mathematics, 21 (2021), pp.~1099--1140.

\bibitem{beskos-geometric-mcmc}
{\sc A.~Beskos, M.~Girolami, S.~Lan, P.~E. Farrell, and A.~M. Stuart}, {\em
  Geometric {MCMC} for infinite-dimensional inverse problems}, Journal of
  Computational Physics, 335 (2017), pp.~327--351.

\bibitem{betancourt2017geometric}
{\sc M.~Betancourt, S.~Byrne, S.~Livingstone, and M.~Girolami}, {\em The
  geometric foundations of {Hamiltonian Monte Carlo}}, Bernoulli, 23 (2017),
  pp.~2257--2298.

\bibitem{PCA-Net}
{\sc K.~Bhattacharya, B.~Hosseini, N.~B. Kovachki, and A.~M. Stuart}, {\em
  Model reduction and neural networks for parametric {PDEs}}, The SMAI journal
  of computational mathematics, 7 (2021), pp.~121--157.

\bibitem{blei2017variational}
{\sc D.~M. Blei, A.~Kucukelbir, and J.~D. McAuliffe}, {\em Variational
  inference: A review for statisticians}, Journal of the American statistical
  Association, 112 (2017), pp.~859--877.

\bibitem{bogachev-gaussian}
{\sc V.~I. Bogachev}, {\em {Gaussian Measures}}, AMS, Providence, 1998.

\bibitem{bogachev2}
{\sc V.~I. Bogachev}, {\em Measure Theory}, vol.~2, Springer, New York, 2007.

\bibitem{bogachev2006nonlinear}
{\sc V.~I. Bogachev and A.~V. Kolesnikov}, {\em Nonlinear transformations of
  convex measures}, Theory of Probability \& Its Applications, 50 (2006),
  pp.~34--52.

\bibitem{bogachev2005triangular}
{\sc V.~I. Bogachev, A.~V. Kolesnikov, and K.~V. Medvedev}, {\em Triangular
  transformations of measures}, Sbornik: Mathematics, 196 (2005), p.~309.

\bibitem{bonneel2015sliced}
{\sc N.~Bonneel, J.~Rabin, G.~Peyr{\'e}, and H.~Pfister}, {\em Sliced and
  {Radon} wasserstein barycenters of measures}, Journal of Mathematical Imaging
  and Vision, 51 (2015), pp.~22--45.

\bibitem{bonnotte2013knothe}
{\sc N.~Bonnotte}, {\em From {K}nothe's rearrangement to {B}renier's optimal
  transport map}, SIAM Journal on Mathematical Analysis, 45 (2013), pp.~64--87.

\bibitem{bunne2022supervised}
{\sc C.~Bunne, A.~Krause, and M.~Cuturi}, {\em Supervised training of
  conditional {Monge} maps}, Advances in Neural Information Processing Systems,
  35 (2022), pp.~6859--6872.

\bibitem{cao2022survey}
{\sc H.~Cao, C.~Tan, Z.~Gao, G.~Chen, P.-A. Heng, and S.~Z. Li}, {\em A survey
  on generative diffusion model}, arXiv:2209.02646,  (2022).

\bibitem{carlier2016vector}
{\sc G.~Carlier, V.~Chernozhukov, and A.~Galichon}, {\em Vector quantile
  regression: an optimal transport approach}, The Annals of Statistics, 44
  (2016), pp.~1165--1192.

\bibitem{carlier2010knothe}
{\sc G.~Carlier, A.~Galichon, and F.~Santambrogio}, {\em From {K}nothe's
  transport to {B}renier's map and a continuation method for optimal
  transport}, SIAM Journal on Mathematical Analysis, 41 (2010), pp.~2554--2576.

\bibitem{chen2018dimension}
{\sc V.~Chen, M.~M. Dunlop, O.~Papaspiliopoulos, and A.~M. Stuart}, {\em
  Dimension-robust {MCMC} in {B}ayesian inverse problems}, arXiv
  preprint:1803.03344,  (2018).

\bibitem{coeurdoux2022learning}
{\sc F.~Coeurdoux, N.~Dobigeon, and P.~Chainais}, {\em Learning optimal
  transport between two empirical distributions with normalizing flows}, in
  Joint European Conference on Machine Learning and Knowledge Discovery in
  Databases, Springer, 2022, pp.~275--290.

\bibitem{stuart-mcmc}
{\sc S.~L. Cotter, G.~O. Roberts, A.~M. Stuart, and D.~White}, {\em {MCMC}
  methods for functions: modifying old algorithms to make them faster},
  Statistical Science, 28 (2013), pp.~424--446.

\bibitem{cui2016dimension}
{\sc T.~Cui, K.~J. Law, and Y.~M. Marzouk}, {\em Dimension-independent
  likelihood-informed {MCMC}}, Journal of Computational Physics, 304 (2016),
  pp.~109--137.

\bibitem{daniels2021score}
{\sc M.~Daniels, T.~Maunu, and P.~Hand}, {\em Score-based generative neural
  networks for large-scale optimal transport}, Advances in neural information
  processing systems, 34 (2021), pp.~12955--12965.

\bibitem{deb2021rates}
{\sc N.~Deb, P.~Ghosal, and B.~Sen}, {\em Rates of estimation of optimal
  transport maps using plug-in estimators via barycentric projections},
  Advances in Neural Information Processing Systems, 34 (2021),
  pp.~29736--29753.

\bibitem{delalande2021quantitative}
{\sc A.~Delalande and Q.~Merigot}, {\em Quantitative stability of optimal
  transport maps under variations of the target measure}, arXiv
  preprint:2103.05934,  (2021).

\bibitem{divol2022optimal}
{\sc V.~Divol, J.~Niles-Weed, and A.-A. Pooladian}, {\em Optimal transport map
  estimation in general function spaces}, arXiv:2212.03722,  (2022).

\bibitem{durmus2017nonasymptotic}
{\sc A.~Durmus and {\'E}.~Moulines}, {\em Nonasymptotic convergence analysis
  for the unadjusted {Langevin} algorithm}, Annals of Applied Probability, 27
  (2017), pp.~1551--1587.

\bibitem{marzouk-opt-map}
{\sc T.~A. El~Moselhy and Y.~M. Marzouk}, {\em {B}ayesian inference with
  optimal maps}, Journal of Computational Physics, 231 (2012), pp.~7815--7850.

\bibitem{embrechts2010bounds}
{\sc P.~Embrechts and G.~Puccetti}, {\em Bounds for the sum of dependent risks
  having overlapping marginals}, Journal of Multivariate Analysis, 101 (2010),
  pp.~177--190.

\bibitem{fan2023quantifying}
{\sc Y.~Fan, H.~Park, and G.~Xu}, {\em Quantifying distributional model risk in
  marginal problems via optimal transport}, arXiv preprint arXiv:2307.00779,
  (2023).

\bibitem{figalli2021invitation}
{\sc A.~Figalli and F.~Glaudo}, {\em An invitation to optimal transport,
  Wasserstein distances, and gradient flows}, EMS, 2021.

\bibitem{fox2012tutorial}
{\sc C.~W. Fox and S.~J. Roberts}, {\em A tutorial on variational {B}ayesian
  inference}, Artificial intelligence review, 38 (2012), pp.~85--95.

\bibitem{garbuno2023bayesian}
{\sc A.~Garbuno-Inigo, T.~Helin, F.~Hoffmann, and B.~Hosseini}, {\em Bayesian
  posterior perturbation analysis with integral probability metrics},
  arXiv:2303.01512,  (2023).

\bibitem{garbuno2020interacting}
{\sc A.~Garbuno-Inigo, F.~Hoffmann, W.~Li, and A.~M. Stuart}, {\em Interacting
  {Langevin} diffusions: Gradient structure and ensemble kalman sampler}, SIAM
  Journal on Applied Dynamical Systems, 19 (2020), pp.~412--441.

\bibitem{genevay2016stochastic}
{\sc A.~Genevay, M.~Cuturi, G.~Peyr{\'e}, and F.~Bach}, {\em Stochastic
  optimization for large-scale optimal transport}, in Advances in Neural
  Information Processing Systems, vol.~29, 2016.

\bibitem{ghosal2022multivariate}
{\sc P.~Ghosal and B.~Sen}, {\em Multivariate ranks and quantiles using optimal
  transport: Consistency, rates and nonparametric testing}, The Annals of
  Statistics, 50 (2022), pp.~1012--1037.

\bibitem{goodfellow2016nips}
{\sc I.~Goodfellow}, {\em {NIPS} 2016 tutorial: Generative adversarial
  networks}, arXiv preprint:1701.00160,  (2016).

\bibitem{goodfellow2016deep}
{\sc I.~Goodfellow, Y.~Bengio, and A.~Courville}, {\em Deep learning}, MIT
  press, 2016.

\bibitem{goodfellow2020generative}
{\sc I.~Goodfellow, J.~Pouget-Abadie, M.~Mirza, B.~Xu, D.~Warde-Farley,
  S.~Ozair, A.~Courville, and Y.~Bengio}, {\em Generative adversarial
  networks}, Communications of the ACM, 63 (2020), pp.~139--144.

\bibitem{grange2023computational}
{\sc D.~Grange, M.~Al-Jarrah, R.~Baptista, A.~Taghvaei, T.~T. Georgiou, and
  A.~Tannenbaum}, {\em Computational optimal transport and filtering on
  {Riemannian} manifolds}, arXiv preprint arXiv:2309.08847,  (2023).

\bibitem{grathwohl2018ffjord}
{\sc W.~Grathwohl, R.~T. Chen, J.~Bettencourt, I.~Sutskever, and D.~Duvenaud},
  {\em {FFJORD}: Free-form continuous dynamics for scalable reversible
  generative models}, in International Conference on Learning Representations,
  2018.

\bibitem{grelaud2009abc}
{\sc A.~Grelaud, C.~P. Robert, J.-M. Marin, F.~Rodolphe, J.-F. Taly, et~al.},
  {\em {ABC} likelihood-free methods for model choice in {G}ibbs random
  fields}, Bayesian Analysis, 4 (2009), pp.~317--335.

\bibitem{gulrajani2017improved}
{\sc I.~Gulrajani, F.~Ahmed, M.~Arjovsky, V.~Dumoulin, and A.~Courville}, {\em
  Improved training of {W}asserstein {GAN}s}, in International Conference on
  Neural Information Processing Systems, 2017.

\bibitem{gutmann2016bayesian}
{\sc M.~U. Gutmann and J.~Corander}, {\em Bayesian optimization for
  likelihood-free inference of simulator-based statistical models}, The Journal
  of Machine Learning Research, 17 (2016), pp.~4256--4302.

\bibitem{hairer2014spectral}
{\sc M.~Hairer, A.~M. Stuart, S.~J. Vollmer, et~al.}, {\em Spectral gaps for a
  {M}etropolis--{H}astings algorithm in infinite dimensions}, The Annals of
  Applied Probability, 24 (2014), pp.~2455--2490.

\bibitem{hosseini2019two}
{\sc B.~Hosseini}, {\em Two {M}etropolis--{H}astings algorithms for posterior
  measures with non-{G}aussian priors in infinite dimensions}, SIAM/ASA Journal
  on Uncertainty Quantification, 7 (2019), pp.~1185--1223.

\bibitem{hosseini2023spectral}
{\sc B.~Hosseini and J.~E. Johndrow}, {\em Spectral gaps and error estimates
  for infinite-dimensional {Metropolis--Hastings} with {non-Gaussian} priors},
  The Annals of Applied Probability, 33 (2023), pp.~1827--1873.

\bibitem{hutter2021minimax}
{\sc J.-C. H{\"u}tter and P.~Rigollet}, {\em Minimax estimation of smooth
  optimal transport maps}, The Annals of Statistics, 49 (2021).

\bibitem{iglesias-subsurface}
{\sc M.~A. Iglesias, K.~Lin, and A.~M. Stuart}, {\em Well-posed {B}ayesian
  geometric inverse problems arising in subsurface flow}, Inverse Problems, 30
  (2014), p.~114001.

\bibitem{irons2021triangular}
{\sc N.~J. Irons, M.~Scetbon, S.~Pal, and Z.~Harchaoui}, {\em Triangular flows
  for generative modeling: Statistical consistency, smoothness classes, and
  fast rates}, in International Conference on Artificial Intelligence and
  Statistics, PMLR, 2022, pp.~10161--10195.

\bibitem{jaini2020tails}
{\sc P.~Jaini, I.~Kobyzev, Y.~Yu, and M.~Brubaker}, {\em Tails of {L}ipschitz
  triangular flows}, in International Conference on Machine Learning, PMLR,
  2020, pp.~4673--4681.

\bibitem{jebara2012machine}
{\sc T.~Jebara}, {\em Machine learning: discriminative and generative},
  vol.~755, Springer Science \& Business Media, 2012.

\bibitem{somersalo}
{\sc J.~Kaipio and E.~Somersalo}, {\em {S}tatistical and {C}omputational
  {I}nverse {P}roblems}, Springer Science \& Business Media, 2005.

\bibitem{kerrigan2023diffusion}
{\sc G.~Kerrigan, J.~Ley, and P.~Smyth}, {\em Diffusion generative models in
  infinite dimensions}, in International Conference on Artificial Intelligence
  and Statistics, PMLR, 2023, pp.~9538--9563.

\bibitem{kingma2021variational}
{\sc D.~Kingma, T.~Salimans, B.~Poole, and J.~Ho}, {\em Variational diffusion
  models}, Advances in neural information processing systems, 34 (2021),
  pp.~21696--21707.

\bibitem{knothe1957contributions}
{\sc H.~Knothe}, {\em Contributions to the theory of convex bodies}, Michigan
  Mathematical Journal, 4 (1957), pp.~39--52.

\bibitem{kobyzev2020normalizing}
{\sc I.~Kobyzev, S.~J. Prince, and M.~A. Brubaker}, {\em Normalizing flows: An
  introduction and review of current methods}, IEEE transactions on pattern
  analysis and machine intelligence, 43 (2020), pp.~3964--3979.

\bibitem{korotin2021neural}
{\sc A.~Korotin, L.~Li, A.~Genevay, J.~M. Solomon, A.~Filippov, and
  E.~Burnaev}, {\em Do neural optimal transport solvers work? a continuous
  {W}asserstein-2 benchmark}, in Advances in Neural Information Processing
  Systems, vol.~34, 2021.

\bibitem{korotin2022kernel}
{\sc A.~Korotin, D.~Selikhanovych, and E.~Burnaev}, {\em Kernel neural optimal
  transport}, in The Eleventh International Conference on Learning
  Representations, 2022.

\bibitem{korotin2023neural}
{\sc A.~Korotin, D.~Selikhanovych, and E.~Burnaev}, {\em Neural optimal
  transport}, arXiv:2201.12220,  (2023).

\bibitem{MGAN}
{\sc N.~Kovachki, R.~Baptista, B.~Hosseini, and Y.~Marzouk}, {\em Conditional
  sampling with monotone {GANs}: from generative models to likelihood-free
  inference}, arXiv:2006.06755,  (2023).

\bibitem{li2023approximation}
{\sc S.~Li and C.~Moosmueller}, {\em Approximation properties of slice-matching
  operators}, arXiv preprint arXiv:2310.10869,  (2023).

\bibitem{li2023measure}
{\sc S.~Li and C.~Moosmueller}, {\em Measure transfer via stochastic slicing
  and matching}, arXiv preprint arXiv:2307.05705,  (2023).

\bibitem{liu2021wasserstein}
{\sc S.~Liu, X.~Zhou, Y.~Jiao, and J.~Huang}, {\em Wasserstein generative
  learning of conditional distribution}, arXiv preprint arXiv:2112.10039,
  (2021).

\bibitem{lueckmann2019likelihood}
{\sc J.-M. Lueckmann, G.~Bassetto, T.~Karaletsos, and J.~H. Macke}, {\em
  Likelihood-free inference with emulator networks}, in Symposium on Advances
  in Approximate Bayesian Inference, PMLR, 2019, pp.~32--53.

\bibitem{mahey2023fast}
{\sc G.~Mahey, L.~Chapel, G.~Gasso, C.~Bonet, and N.~Courty}, {\em Fast optimal
  transport through sliced wasserstein generalized geodesics}, arXiv preprint
  arXiv:2307.01770,  (2023).

\bibitem{makkuva2020optimal}
{\sc A.~Makkuva, A.~Taghvaei, S.~Oh, and J.~Lee}, {\em Optimal transport
  mapping via input convex neural networks}, in International Conference on
  Machine Learning, PMLR, 2020, pp.~6672--6681.

\bibitem{manole2021plugin}
{\sc T.~Manole, S.~Balakrishnan, J.~Niles-Weed, and L.~Wasserman}, {\em Plugin
  estimation of smooth optimal transport maps}, arXiv preprint
  arXiv:2107.12364,  (2021).

\bibitem{marzouk2016sampling}
{\sc Y.~Marzouk, T.~Moselhy, M.~Parno, and A.~Spantini}, {\em Sampling via
  measure transport: An introduction}, Handbook of Uncertainty Quantification,
  (2016), pp.~1--41.

\bibitem{merigot2020quantitative}
{\sc Q.~M{\'e}rigot, A.~Delalande, and F.~Chazal}, {\em Quantitative stability
  of optimal transport maps and linearization of the 2-wasserstein space}, in
  International Conference on Artificial Intelligence and Statistics, PMLR,
  2020, pp.~3186--3196.

\bibitem{mirza2014conditional}
{\sc M.~Mirza and S.~Osindero}, {\em Conditional generative adversarial nets},
  arXiv preprint arXiv:1411.1784,  (2014).

\bibitem{muzellec2019subspace}
{\sc B.~Muzellec and M.~Cuturi}, {\em Subspace detours: Building transport
  plans that are optimal on subspace projections}, Advances in Neural
  Information Processing Systems, 32 (2019).

\bibitem{ng2002discriminative}
{\sc A.~Y. Ng and M.~I. Jordan}, {\em On discriminative vs. generative
  classifiers: A comparison of logistic regression and naive {B}ayes}, in
  Advances in neural information processing systems, 2002, pp.~841--848.

\bibitem{onken2021ot}
{\sc D.~Onken, S.~Wu~Fung, X.~Li, and L.~Ruthotto}, {\em {OT}-flow: Fast and
  accurate continuous normalizing flows via optimal transport}, in Proceedings
  of the AAAI Conference on Artificial Intelligence, vol.~35, 2021.

\bibitem{papamakarios2019normalizing}
{\sc G.~Papamakarios, E.~T. Nalisnick, D.~J. Rezende, S.~Mohamed, and
  B.~Lakshminarayanan}, {\em Normalizing flows for probabilistic modeling and
  inference.}, Journal of Machine Learning Research, 22 (2021), pp.~1--64.

\bibitem{papamakarios2019sequential}
{\sc G.~Papamakarios, D.~Sterratt, and I.~Murray}, {\em Sequential neural
  likelihood: Fast likelihood-free inference with autoregressive flows}, in
  International Conference on Artificial Intelligence and Statistics, 2019.

\bibitem{pardo2018statistical}
{\sc L.~Pardo}, {\em Statistical inference based on divergence measures}, CRC
  press, 2018.

\bibitem{peyre2019computational}
{\sc G.~Peyr{\'e} and M.~Cuturi}, {\em Computational optimal transport: With
  applications to data science}, Foundations and Trends{\textregistered} in
  Machine Learning, 11 (2019), pp.~355--607.

\bibitem{rabin2010geodesic}
{\sc J.~Rabin, G.~Peyr{\'e}, and L.~D. Cohen}, {\em Geodesic shape retrieval
  via optimal mass transport}, in Computer Vision--ECCV 2010: 11th European
  Conference on Computer Vision, Heraklion, Crete, Greece, September 5-11,
  2010, Proceedings, Part V 11, Springer, 2010, pp.~771--784.

\bibitem{rabin2012wasserstein}
{\sc J.~Rabin, G.~Peyr{\'e}, J.~Delon, and M.~Bernot}, {\em Wasserstein
  barycenter and its application to texture mixing}, in Scale Space and
  Variational Methods in Computer Vision: Third International Conference, SSVM
  2011, Ein-Gedi, Israel, May 29--June 2, 2011, Revised Selected Papers 3,
  Springer, 2012, pp.~435--446.

\bibitem{rahman2022generative}
{\sc M.~A. Rahman, M.~A. Florez, A.~Anandkumar, Z.~E. Ross, and
  K.~Azizzadenesheli}, {\em Generative adversarial neural operators},
  Transactions on Machine Learning Research,  (2022).

\bibitem{ramgraber2022ensemble}
{\sc M.~Ramgraber, R.~Baptista, D.~McLaughlin, and Y.~Marzouk}, {\em Ensemble
  transport smoothing--part 1: unified framework}, arXiv preprint
  arXiv:2210.17000,  (2022).

\bibitem{ramgraber2022ensemble-2}
{\sc M.~Ramgraber, R.~Baptista, D.~McLaughlin, and Y.~Marzouk}, {\em Ensemble
  transport smoothing--part 2: nonlinear updates}, arXiv preprint
  arXiv:2210.17435,  (2022).

\bibitem{ray2022efficacy}
{\sc D.~Ray, H.~Ramaswamy, D.~V. Patel, and A.~A. Oberai}, {\em The efficacy
  and generalizability of conditional {GANs} for posterior inference in
  physics-based inverse problems}, arXiv preprint:2202.07773,  (2022).

\bibitem{rezende2015variational}
{\sc D.~Rezende and S.~Mohamed}, {\em Variational inference with normalizing
  flows}, in International Conference on Machine Learning, 2015.

\bibitem{robert1999monte}
{\sc C.~P. Robert, G.~Casella, and G.~Casella}, {\em Monte Carlo statistical
  methods}, vol.~2, Springer, 1999.

\bibitem{rosenblatt1952remarks}
{\sc M.~Rosenblatt}, {\em Remarks on a multivariate transformation}, The annals
  of mathematical statistics, 23 (1952), pp.~470--472.

\bibitem{rout2021generative}
{\sc L.~Rout, A.~Korotin, and E.~Burnaev}, {\em Generative modeling with
  optimal transport maps}, in International Conference on Learning
  Representations, 2021.

\bibitem{ruschendorf1991bounds}
{\sc L.~R{\"u}schendorf}, {\em Bounds for distributions with multivariate
  marginals}, Lecture Notes-Monograph Series,  (1991), pp.~285--310.

\bibitem{saharia2022image}
{\sc C.~Saharia, J.~Ho, W.~Chan, T.~Salimans, D.~J. Fleet, and M.~Norouzi},
  {\em Image super-resolution via iterative refinement}, IEEE Transactions on
  Pattern Analysis and Machine Intelligence, 45 (2022), pp.~4713--4726.

\bibitem{santambrogio2015optimal}
{\sc F.~Santambrogio}, {\em Optimal transport for applied mathematicians},
  Springer, 2015.

\bibitem{siahkoohi2021preconditioned}
{\sc A.~Siahkoohi, G.~Rizzuti, M.~Louboutin, P.~A. Witte, and F.~J. Herrmann},
  {\em Preconditioned training of normalizing flows for variational inference
  in inverse problems}, arXiv preprint:2101.03709,  (2021).

\bibitem{siahkoohi2023reliable}
{\sc A.~Siahkoohi, G.~Rizzuti, R.~Orozco, and F.~J. Herrmann}, {\em Reliable
  amortized variational inference with physics-based latent distribution
  correction}, Geophysics, 88 (2023), pp.~R297--R322.

\bibitem{sohl2015deep}
{\sc J.~Sohl-Dickstein, E.~Weiss, N.~Maheswaranathan, and S.~Ganguli}, {\em
  Deep unsupervised learning using nonequilibrium thermodynamics}, in
  International conference on machine learning, PMLR, 2015, pp.~2256--2265.

\bibitem{song2020improved}
{\sc Y.~Song and S.~Ermon}, {\em Improved techniques for training score-based
  generative models}, Advances in neural information processing systems, 33
  (2020), pp.~12438--12448.

\bibitem{spantini2019coupling}
{\sc A.~Spantini, R.~Baptista, and Y.~Marzouk}, {\em Coupling techniques for
  nonlinear ensemble filtering}, SIAM Review (In Press),  (2019).

\bibitem{sprungk2020local}
{\sc B.~Sprungk}, {\em On the local lipschitz stability of {Bayesian} inverse
  problems}, Inverse Problems, 36 (2020), p.~055015.

\bibitem{stuart-acta-numerica}
{\sc A.~M. Stuart}, {\em Inverse problems: a {B}ayesian perspective}, Acta
  Numerica, 19 (2010), pp.~451--559.

\bibitem{taghvaei2022optimal}
{\sc A.~Taghvaei and B.~Hosseini}, {\em An optimal transport formulation of
  {Bayes’} law for nonlinear filtering algorithms}, in 2022 IEEE 61st
  Conference on Decision and Control (CDC), IEEE, 2022, pp.~6608--6613.

\bibitem{trippe2018conditional}
{\sc B.~L. Trippe and R.~E. Turner}, {\em Conditional density estimation with
  {Bayesian} normalising flows}, arXiv preprint arXiv:1802.04908,  (2018).

\bibitem{uscidda2023monge}
{\sc T.~Uscidda and M.~Cuturi}, {\em The {Monge Gap}: A regularizer to learn
  all transport maps}, arXiv:2302.04953,  (2023).

\bibitem{villani-OT}
{\sc C.~Villani}, {\em Optimal transport: Old and new}, Springer, 2009.

\bibitem{wang2022minimax}
{\sc S.~Wang and Y.~Marzouk}, {\em On minimax density estimation via measure
  transport}, arXiv:2207.10231,  (2022).

\bibitem{wang2023efficient}
{\sc Z.~O. Wang, R.~Baptista, Y.~Marzouk, L.~Ruthotto, and D.~Verma}, {\em
  Efficient neural network approaches for conditional optimal transport with
  applications in {Bayesian} inference}, arXiv preprint arXiv:2310.16975,
  (2023).

\bibitem{winkler2019learning}
{\sc C.~Winkler, D.~Worrall, E.~Hoogeboom, and M.~Welling}, {\em Learning
  likelihoods with conditional normalizing flows}, arXiv preprint
  arXiv:1912.00042,  (2019).

\bibitem{yin2023solving}
{\sc Z.~Yin, R.~Orozco, M.~Louboutin, and F.~J. Herrmann}, {\em Solving
  multiphysics-based inverse problems with learned surrogates and constraints},
  Advanced Modeling and Simulation in Engineering Sciences, 10 (2023), p.~14.

\bibitem{Yosida1965}
{\sc K.~Yosida}, {\em Functional Analysis}, Springer-Verlag, New York, 6~ed.,
  1980.

\bibitem{zech2022sparse-I}
{\sc J.~Zech and Y.~Marzouk}, {\em Sparse approximation of triangular
  transports, part i: The finite-dimensional case}, Constructive Approximation,
  55 (2022), pp.~919--986.

\bibitem{zech2022sparse-II}
{\sc J.~Zech and Y.~Marzouk}, {\em Sparse approximation of triangular
  transports, part ii: The infinite-dimensional case}, Constructive
  Approximation, 55 (2022), pp.~987--1036.

\bibitem{zhang2018advances}
{\sc C.~Zhang, J.~B{\"u}tepage, H.~Kjellstr{\"o}m, and S.~Mandt}, {\em Advances
  in variational inference}, IEEE transactions on pattern analysis and machine
  intelligence, 41 (2018), pp.~2008--2026.

\bibitem{zhou2022deep}
{\sc X.~Zhou, Y.~Jiao, J.~Liu, and J.~Huang}, {\em A deep generative approach
  to conditional sampling}, Journal of the American Statistical Association,
  (2022), pp.~1--12.

\end{thebibliography}

\end{document}